\documentclass[12pt]{article}
\usepackage{graphicx,subfigure}
\usepackage{color,latexsym,amsfonts,amssymb}
\usepackage{bm}     
\usepackage{amsthm} 
\usepackage{amsmath}
\usepackage{booktabs} 
\usepackage{boxedminipage}  
\usepackage{algorithm,algpseudocode}  
\usepackage{subfigure}

\DeclareMathOperator*{\argmin}{argmin}

\newtheorem{theorem}{Theorem}
\newtheorem{corollary}{Corollary}
\newtheorem{lemma}{Lemma}
\newtheorem{proposition}{Proposition}
\newtheorem{assumption}{Assumption}

\newcommand{\dif}{\mathrm{d}}
\title{A $c/\mu$-Rule for Service Resource Allocation in Group-Server Queues}
\author{Li Xia \thanks{Li Xia is with the Center for Intelligent and Networked Systems (CFINS),
Department of Automation, TNList, Tsinghua University, Beijing
100084, China (e-mail: xial@tsinghua.edu.cn).}, Zhe George Zhang
\thanks{Zhe George Zhang is with the Department of Decision Sciences, Western Washington University, Bellingham, Washington 98225;
and Faculty of Business Administration, Simon Fraser University,
Burnaby, British Columbia V5A 1S6, Canada (e-mail:
george.zhang@wwu.edu).}, Quan-Lin Li \thanks{Quan-Lin Li is with the
School of Economics and Management Sciences, Yanshan University,
Qinhuangdao 066004, China (e-mail: liquanlin@tsinghua.edu.cn).},
Peter W. Glynn
\thanks{Peter W. Glynn is with the Department of Management Science and Engineering, Stanford University, Stanford, CA 94305, USA (e-mail:
glynn@stanford.edu).}}

\date{}

\begin{document}
\maketitle
\begin{abstract}
This paper addresses a dynamic on/off server scheduling problem in a
queueing system with Poisson arrivals and $K$ groups of exponential
servers. Servers within each group are homogeneous and between
groups are heterogeneous. A scheduling policy prescribes the number
of working servers in each group at every state $n$ (number of
customers in the system). Our goal is to find the optimal policy to
minimize the long-run average cost, which consists of an increasing
convex holding cost for customers and a linear operating cost for
working servers. We prove that the optimal policy has monotone
structures and quasi bang-bang control forms. Specifically, we find
that the optimal policy is indexed by the value of $c - \mu G(n)$,
where $c$ is the operating cost rate, $\mu$ is the service rate, and
$G(n)$ is called \emph{perturbation realization factor}. Under a
reasonable condition of scale economies, we prove that the optimal
policy obeys a so-called \emph{$c$/$\mu$-rule}. That is, the servers
with smaller $c$/$\mu$ should be turned on with higher priority.
With the monotone property of $G(n)$, we further prove that the
optimal policy has a multi-threshold structure when the
$c$/$\mu$-rule is applied. Numerical experiments demonstrate that
the $c/\mu$-rule has a good scalability and robustness.
\end{abstract}


\section{Introduction}\label{section_intro}
Queueing models are often formulated to study stochastic congestion
problems in manufacturing and service systems, computer and
communication networks, social economics, etc. Research on queueing
models is spreading from performance evaluation to performance
optimization for system design and control. Queueing phenomena is
caused by limited service resource. How to efficiently allocate the
service resource is a fundamental issue in queueing systems, which
continually attracts attention from the operational research
community, such as several papers published in \emph{Oper. Res.} in
recent volumes \cite{Dieker17,Tsitsiklis17}.

In practice, there exists a category of queueing systems with
multiple servers that provide homogeneous service at different
service rates and costs. These servers can be categorized into
different groups. Servers in the same group have the same service
rate and cost rate, while those in different groups are
heterogeneous in these two rates. Customers wait in line when all
the servers are busy. Running a server (keeping a server on) will
incur operating cost and holding a customer will incur waiting cost.
There exists a tradeoff between these costs. Keeping more servers on
will increase the operating cost but decrease the holding cost. The
holding cost is increasingly convex in queue length and the
operating cost is linear in the number of working servers. The
system controller can dynamically turn on or off servers according
to different backlogs (queue lengths) such that the system long-run
average cost can be minimized. We call such queueing models
\emph{group-server queues} \cite{Li17}. Note that any queueing
system with heterogeneous servers can be considered as a
group-server queue if the servers with the same service cost and
rate are grouped as a class. Following are some examples that
motivate the group-server queueing model.

\begin{itemize}
\item
Multi-tier storage systems: As illustrated in
Fig.~\ref{fig_multier}, such a multi-tier storage architecture is
widely used in intelligent storage systems, where different storages
are structured as multiple tiers and data are stored and migrated
according to their hotness (access frequency) \cite{Wang14,Zhang10}.
Solid state drive (SSD), hard disk drive (HDD), and cloud storage
are organized in a descending order of their speeds and costs. A
group-server queue fits such a system and can be used to study the
system performance, such as the response time of I/O requests. It is
an interesting topic to find the optimal architecture and scheduling
of I/O requests so that the desired system throughput is achieved at
a minimum cost.

\item
Clustered computing systems: As illustrated in
Fig.~\ref{fig_clustercomp}, the computing facilities of a server
farm are organized in clusters. Computers in different clusters have
different performance and power consumption. For example, high
performance computer (HPC) has greater processing rate and more
power consumption. Computing jobs can be scheduled and migrated
among computers. Energy efficiency is one of the key metrics for
evaluating the performance of data centers. Power management policy
aims to dynamically schedule servers' working states (e.g., high/low
power, or sleep) according to workloads such that power consumption
and processing rate can be traded-off in an optimal way
\cite{Gandhi10,Kant09}.

\item
Human staffed service systems: One example is a call center that
might have several groups of operators (customer representatives) in
different locations (or different countries). Depending on the
demand level, the number of operator groups attending calls may be
dynamically adjusted. The service efficiencies and operating costs
of these groups can be different although operators in each group
can be homogeneous in these two aspects. Another example is the
operation of a food delivery company, such as GrubHub in the US or
Ele.me in China, which has several restaurant partners with good
reputation. During the high demand period, the limited number of its
own delivery drivers (servers in group 1) may not be able to deliver
food orders to customers in a promised short time (due to long
queues). Thus, the company can share part of their delivery service
demands with another less reputable food delivery company who also
owns a set of drivers (servers in group 2).
\end{itemize}

\begin{figure}[htbp]
\centering \subfigure[A multi-tier storage system.]
{\includegraphics[width=0.45\columnwidth]{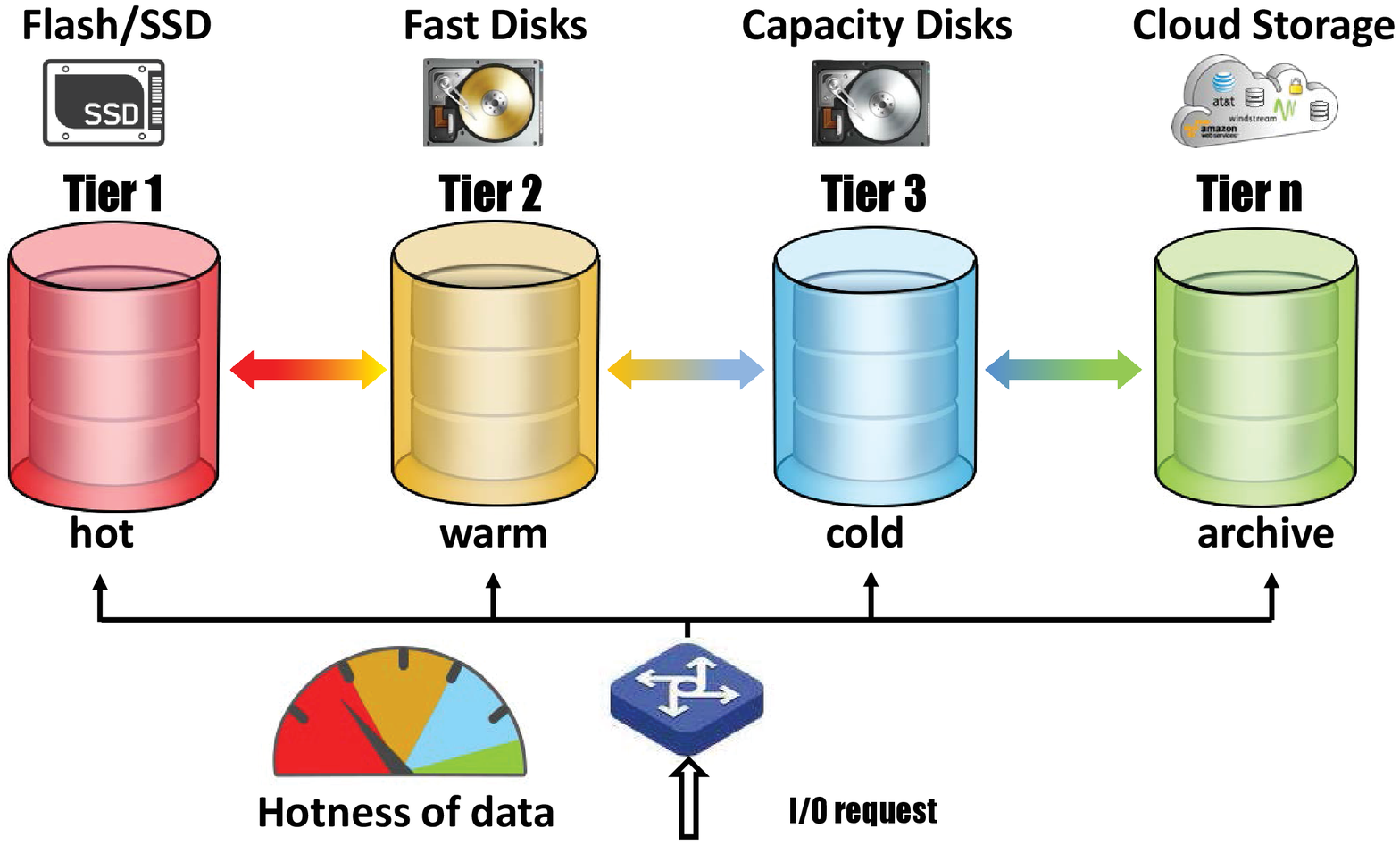}\label{fig_multier}}
\centering \subfigure[A clustered computing system.]
{\includegraphics[width=0.45\columnwidth]{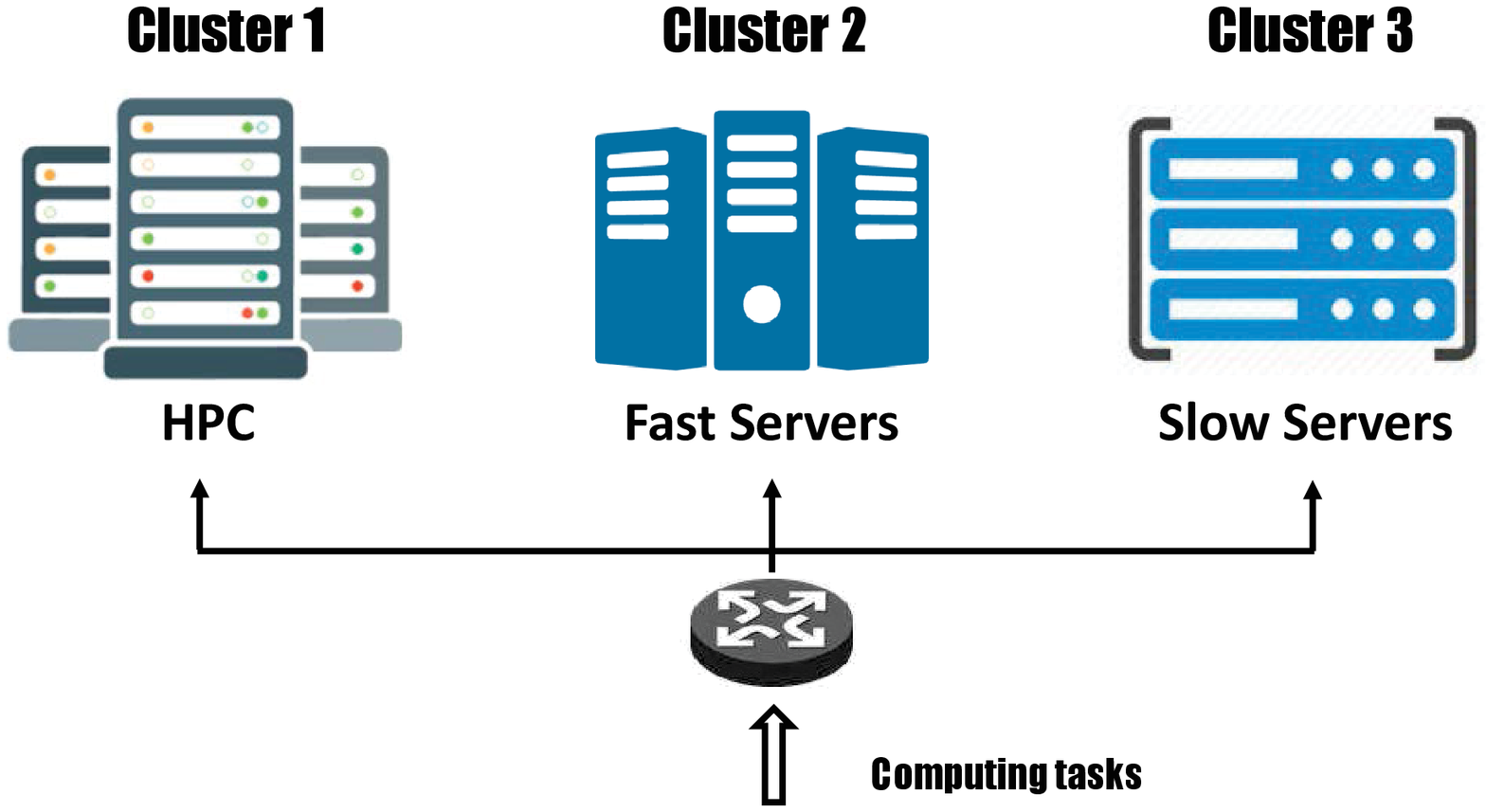}\label{fig_clustercomp}}
\caption{Motivations of group-server queues.}\label{fig_motivation}
\end{figure}

Similar resource allocation problems may exist in other systems such
as clustered wireless sensor networks \cite{Kumar09}, tiered web
site systems \cite{Urgaonkar05}, tiered tolling systems
\cite{Hua16}, etc. The common features of these problems can be well
captured by the group-server queue. How to efficiently schedule the
server groups to optimize the targeted performance metrics is an
important issue for both practitioners and queueing researchers. To
address this issue, we focus on finding the optimal on/off server
scheduling policy in a group-server queue to minimize the long-run
average cost.

\subsection{Related Research}
Service resource allocation problems in queueing systems are widely
studied in the literature. One stream of research focuses on the
\emph{service rate control} which aims to find the optimal service
rates such that the system average cost (holding cost plus operating
cost) can be minimized. This type of problems are mainly motivated
by improving the operational efficiency of computer and
telecommunication systems. For a server with a fixed service rate,
turning it on or off can be considered as a service rate control of
full or zero service rate. The optimality of threshold type policy,
such as $N$-policy, $D$-policy, and $T$-policy, has been studied in
single-server queueing systems with fixed service rate
\cite{Balachandran75,Heyman68,Heyman77}. Then, further studies
extend single-server systems to multi-server networks, such as
cyclic queues or tandem queues \cite{Rosberg82,Stidham93,Weber87},
where the optimality of bang-bang control or threshold type policy
is studied. Note that the bang-bang control means that even for the
case where the service rate can be chosen in a finite range, the
optimal rate is always at either zero or the maximum rate depending
on the system size (threshold). For complicated queueing networks,
such as Jackson networks, it has been proved that the bang-bang
control is optimal when the cost function has a linear form to
service rates, using the techniques of linear programming by Yao and
Schechner \cite{Yao89} and derivative approach by Ma and Cao
\cite{Ma94}, respectively. Recent works by Xia et al. further extend
such optimality structure from a linear cost function to a concave
one \cite{Xia13,Xia15}. Another line of studying the service rate
control problem is from a game theoretic viewpoint
\cite{Hassin03,Xia14}. We cannot enumerate all service rate control
studies due to space limit of this paper. A common feature of the
past studies is to characterize the structure of optimal rate
control policy in a variety of queueing systems.

The tradeoff between holding cost and operating cost is also a major
issue in some service systems with human servers. Thus, there exist
many studies on \emph{server scheduling problems} (or called
\emph{staffing problems}) which aim to dynamically adjust the number
of servers to minimize the average holding and operating costs. An
early work is by Yadin and Naor who study the dynamic on/off
scheduling policy of a server in an $M/G/1$ queue with a non-zero
setup time \cite{Yadin63}. Many other related works can be found in
this area and we just name a few \cite{Bell80,Fu00,Sobel69,Yoo96}.
To control the customer waiting time and improve the server
utilization, Zhang studies a congestion-based staffing (CBS) policy
for a multi-server service system motivated by the US-Canada
border-crossing stations \cite{Zhang09}. Servers in these studies
are assumed to be homogeneous. The CBS policy has a two-threshold
structure and can be considered as a generalization of the
multi-server queue with server vacations, which is an important
class of queueing models \cite{Tian06}.

\emph{Job assignment problem} in heterogenous servers is closely
related to the on/off server scheduling problem treated in this
paper. It has one queue and multiple servers. It focuses on optimal
scheduling of homogeneous jobs at heterogeneous servers with
different service rates and/or operating costs. In one class of
problems, the objective is to minimize the average waiting time of
jobs under the assumption that only holding cost is relevant and the
operating cost is sunk (i.e., not considered). In addition, when a
job is assigned to a server, it cannot be reassigned to other faster
(more desirable) server which becomes available later. Such a
problem is also called a \emph{slow-server problem} and can be used
to study the job routing policy in computer systems. One pioneering
work is by Lin and Kumar and they study the optimal control of an
M/M/2 queue with two heterogeneous servers. They prove that the
faster server should be always utilized while the slower one should
be utilized only when the queue length exceeds a computable
threshold value \cite{Lin84, Walrand84}. For the case with more than
two servers, it is shown that the fastest available server (FAS)
rule is optimal \cite{Millhiser16}. However, for other servers
except for FAS, it is difficult to directly extend the single
threshold (two-server system) to the multi-threshold optimality
(more than two-server system), although it looks intuitive. This is
because the system state becomes higher dimensional that makes the
dynamic programming based analysis very complicated. Weber proposes
a conjecture about the threshold optimality for multiple
heterogenous servers and shows that the threshold may depend on the
state of slower servers \cite{Weber93}. Rykov proves this conjecture
using dynamic programming \cite{Rykov01} and Luh and Viniotis prove
it using linear programming \cite{Luh02}, but their proofs are
opaque or incomplete \cite{deVericourt06}. Armony and Ward further
study a fair dynamic routing problem, in which the customer average
waiting time is minimized at the constraint of a fair idleness
portion of servers \cite{Armony10,Ward13}. Constrained Markov
decision processes and linear programming models are utilized to
characterize that the optimal routing policy asymptotically has a
threshold type in a limit regime with many servers \cite{Armony10}.
There are numerous studies on the slow-server problem from various
perspectives, which are summarized in \cite{Akgun14,Hassin15,Xu93}.

When job reassignment (also called job migration) is allowed, the
slow-server problem becomes trivial since it is optimal to always
assign jobs to available fastest servers. However, when the server
operating cost (such as power consumption) is considered, the job
assignment problem is not trivial even with job migration allowed.
In fact, both holding and operating costs should be considered in
practical systems, such as energy efficient data centers or cloud
computing facilities \cite{Fu16}. Akgun et al. give a comprehensive
study on this problem \cite{Akgun14}. They utilize the duality
between the individually optimal and socially optimal policies
\cite{Hassin85,Xu93} to prove the threshold optimality of
heterogenous servers for a clearing system (no arrivals) with or
without reassignment. They also prove the threshold optimality for
the less preferred server in a two-server system with customer
arrivals. It is shown that the preference of servers depends on not
only their service rates, but also the usage costs (operating
costs), holding costs, arrival rate, and the system state.

Under a cost structure with both holding and operating costs, the
job assignment problem for heterogeneous servers with customer
arrivals and job migration can be viewed as an equivalence to our
on/off server scheduling problem in a group-server queue. In this
paper, we characterize the structure of the optimal policy which can
significantly simplify the computation of the parameters of the
optimal on/off server schedule. In general, under the optimal
policy, a server group will not be turned on only if the ratio of
operating cost rate $c$ to service processing rate $\mu$ is smaller
than a computable quantity $G(n)$, called perturbation realization
factor. The perturbation realization factor depends on the number of
customers in the system (system state), the arrival rate, and the
cost function. We call this type of policy an \emph{index policy}
and it has a form of state-dependent multi-threshold. The term of
state-dependent means that the preference rankings of groups (the
order of server groups to be turned on) will change from one state
to another. However, under a reasonable condition of server group's
scale economies, the optimal index policy is reduced to a
state-independent multi-threshold policy, called \emph{the
$c/\mu$-rule}. This simple rule is easy to implement in practice and
complements the well-known $c\mu$-rule for polling queues. In a
polling queueing system, a single server serves multiple classes of
customers which form multiple queues and a polling policy prescribes
which queue to serve by the single server. In a group-server
queueing system, heterogeneous servers grouped into multiple classes
serve homogeneous customers (a single queue) and an on/off server
schedule prescribes which server group is turned on to serve the
single queue. Note that the ``$c$" in the $c \mu$-rule is the
customer waiting cost rate, while the ``$c$" in the $c/\mu$-rule is
the server operating cost rate. Due to the difference in the cost
rate $c$, it is intuitive that the customer class with the highest
$c \mu$ value should be served first and the server group with the
lowest $c /\mu$ value should be utilized first. Although these
results are kind of intuitive, the $c \mu$-rule was studied long
time ago but the $c/\mu$-rule was not well established until this
paper. This may be because of the more complexity caused by the
heterogeneous server system. Note that although the $c/\mu$ rankings
order the server's operating cost per unit of service rate,
different service rates impact the customer holding cost
differently. In contrast, in a polling queueing system, the only
cost difference between polling two different queues is the $c \mu$,
the holding cost moving out of the system per unit of time. Thus, a
static $c \mu$-rule can be established as an optimal policy to
minimize the system average cost.

The early work of the $c\mu$-rule can be traced back to Smith's
paper in 1956 under a deterministic and static setting
\cite{Smith56}. Under the $c\mu$-rule, the queue with larger $c\mu$
value should be served with higher priority. This rule is very
simple and easy to implement in practice. It stimulates numerous
extensions in the literature
\cite{Hirayama89,Kebarighotbi11,Kilmov74,Nain94,VanMieghem03}. Many
works aim to study similar properties to the $c\mu$-rule under
various queueing systems and assumptions. For example, Baras et al.
study the optimality of the $c\mu$-rule from 2 to $K$ queues with
linear costs and geometric service requirement
\cite{Baras85,Baras85b}, and Buyukkoc et al. revisit the proof of
the $c\mu$-rule in a simple way \cite{Buyukkoc85}. Van Mieghem
studies the asymptotic optimality of a generalized version of the
$c\mu$-rule with convex holding costs in heavy traffic settings
\cite{VanMieghem95}. This work is then extended by Mandelbaum and
Stolyar to a network topology \cite{Mandelbaum04}. Atar et al.
further study another generalized version called the $c\mu/\theta$
rule in an abandonment queue where $\theta$ is the abandonment rate
of impatient customers \cite{Atar10}. Recently, Saghafian and Veatch
study the $c\mu$-rule in a two-tiered queue \cite{Saghafian16}. In
contrast to the extensive studies on the $c\mu$-rule in the
literature, there are few studies on the $c/\mu$ rule for the
resource allocation in a single queue with heterogeneous servers.

\subsection{Our Contributions}
One of the significant differences between our work and relevant
studies in the literature is that the servers in our model are
heterogeneous and categorized into multiple groups, which makes the
model more general but more complex. Most heterogenous server models
in the literature may be viewed as a special case of our model, in
which each group has only one server. Thus, our model is more
applicable to large scale service systems such as data centers.
Moreover, we assume that there is an unlimited waiting room for
customers, which means that the dynamic policy is over an infinite
state space. To find the optimal policy over the infinite state
space is difficult. Thus, we aim to characterize the structure of
the optimal policy. While the holding cost in job assignment
problems is usually assumed to be linear, we assume the holding cost
can be any increasing convex function (a generalization of linear
function). We formulate this service resource allocation problem in
a group-server queue as a Markov decision process (MDP). Unlike the
traditional MDP approach, we utilize the sensitivity-based
optimization theory to characterize the structure of the optimal
policy and develop efficient algorithms to compute the optimal
policy and thresholds.

The main contribution of this paper can be summarized in the
following aspects.

\begin{itemize}
\item Index policy: The server preference (priority of being turned on) is
determined by an index $c-\mu G(n)$, where $G(n)$ is the
perturbation realization factor and it is computable and
state-dependent. Servers with more negative value of $c-\mu G(n)$
have more preference to be turned on. Servers with positive $c-\mu
G(n)$ should be kept off. The value of $G(n)$ will affect the
preference order of servers and depends on $n$, arrival rate, and
cost functions. We prove the optimality of this index policy and
show that $G(n)$ plays a fundamental role in determining the optimal
index policy.

\item The $c/\mu$-rule: Under the condition of scale economies for
server groups, the preference of servers can be determined by their
$c/\mu$ values, instead of $c-\mu G(n)$. Thus, the preference order
of servers is independent of $n$, arrival rate, and cost functions.
The server's on/off scheduling policy becomes the $c/ \mu$ rule.
Under this rule, the server with smaller $c/\mu$ should be turned on
with higher priority. Searching the optimal policy over an
infinite-dimensional mapping space is reduced to searching the
optimal multiple thresholds. Multi-threshold policy is easier to be
implemented in practice and robust.

\item Optimality structures: With the performance difference
formula, we derive a necessary and sufficient condition of the
optimal policy. The optimality of quasi bang-bang control is also
established. The monotone and convexity properties of performance
potentials and perturbation realization factors, which are
fundamental quantities during optimization, are established. With
these properties, the optimality of index policy and the $c/\mu$
rule is proved. The structure of optimal policy is characterized
well and the optimization complexity is reduced significantly.

\end{itemize}

Besides the theoretical contributions in the above aspects, using
the performance difference formula, we decompose the original
problem into an infinite number of integer linear programs. Based on
the structure of the optimal policy, we develop iterative algorithms
to find the optimal index policy or optimal multi-threshold policy.
Here, the $c/\mu$-rule can be utilized to simplify the search
algorithms significantly. These algorithms are similar to the policy
iteration in the traditional MDP theory and their performance is
demonstrated by numerical examples.

\subsection{Paper Organization}
The rest of the paper is organized as follows. In
Section~\ref{section_model}, a model of group-server queue is
developed to capture the heterogeneity of servers. An optimization
problem is formulated to determine the on/off server scheduling
policy of cost minimization. The analysis is presented in
Section~\ref{section_result}, where the structure of optimal index
policy is characterized based on the perturbation realization factor
of server groups. In Section~\ref{section_rule}, we derive the
$c/\mu$-rule and study the optimality of multi-threshold policy
under the condition of scale economies. In
Section~\ref{section_numerical}, we conduct numerical experiments to
gain the managerial insights and to show the efficiency of our
approach. Finally, the paper is concluded in
Section~\ref{section_conclusion} with a summary.

\section{Optimization Problem in Group-Server Queues}\label{section_model}
In this section, we describe the service resource allocation problem
in a group-server queue model. This model can be used to represent a
waiting line with heterogeneous servers classified into a finite
number of groups, which can also be called parallel-server systems
in previous studies \cite{Armony10}. Servers are homogeneous within
the group and are heterogeneous between groups. A group-server queue
is shown in Fig.~\ref{fig_GSqueue} and described as follows.

\begin{figure}[htbp]
\centering
\includegraphics[width=0.55\columnwidth]{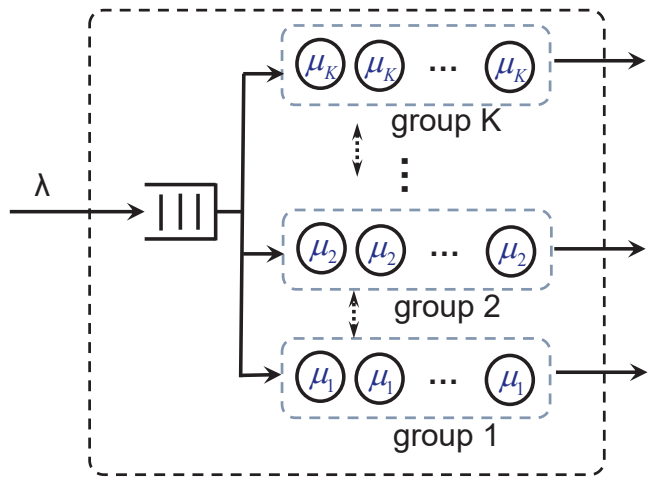}
\caption{An example of group-server queue model, where servers are
in parallel.}\label{fig_GSqueue}
\end{figure}

Customers arrive to a service station with multiple groups of
servers according to a Poisson process with rate $\lambda$. The
waiting room is infinite and the service discipline is
first-come-first-serve (FCFS). The service times of each server are
assumed to be independent and exponentially distributed. The
heterogeneous servers are classified into $K$ groups. Each group has
$M_k$ servers, which can be turned on or off, $k=1,2,\cdots,K$. When
a server in group $k$ is turned on, it will work at service rate
$\mu_k$ and consume an operating cost $c_k$ per unit of time.
Servers in the same group are homogeneous, i.e., they have the same
service rate $\mu_k$ and cost rate $c_k$, $k=1,2,\cdots,K$. Servers
in different groups are heterogeneous in $\mu_k$ and $c_k$. We
assume that servers in different groups offer the same service,
i.e., customers are homogeneous. In general, services offered by
different groups may be different and the connection of groups may
be cascaded, or even interconnected. Such a setting can be called a
\emph{group-server queueing network}. When a working server has to
be turned off, the customer being served at that server is
interrupted and transferred to the waiting room or another idle
server if available. Due to the memoryless property of the service
time, such an interruption has no effect on customer's remaining
service time.

The system state $n$ is defined as the number of customers in the
system (including those in service). The on/off status of servers
need not be included in the system state because free customer
migrations among servers are allowed in the model. Thus, the state
space is the nonnegative integer set $\mathbb N$, which is infinite.
At each state $n \in \mathbb N$, we determine the number of working
servers in each group, which can be represented by a $K$-dimensional
row vector as
\begin{equation}
\bm m:=(m_1,m_2,\cdots,m_K),
\end{equation}
where $m_k$ is the number of working servers in group $k$, i.e.,
$m_k \in \mathbb Z_{[0,M_k]}$, $k=1,2,\cdots,K$. We call $\bm
m$ the scheduling action at state $n$, according to the terminology
of MDPs. Thus, the action space is defined as
\begin{equation}
\mathbb M := \mathbb Z_{[0,M_1]} \times \mathbb Z_{[0,M_2]} \times
\cdots \times \mathbb Z_{[0,M_K]},
\end{equation}
where $\times$ is the Cartesian product. We assume that the system
has reached a steady state under a condition to be specified later
in Proposition~\ref{pro2}. Therefore, a stationary scheduling policy
$d$ is defined as a mapping from the infinite state space $\mathbb
N$ to the finite action space $\mathbb M$, i.e., $d: \mathbb N
\rightarrow \mathbb M$. If $d$ is determined, we will adopt action
$d(n)$ at every state $n$ and $d(n,k)$ is the number of working
servers of group $k$, where $n \in \mathbb N$ and $k=1,2,\cdots,K$.
All possible $d$'s form the policy space $\mathcal D$, which is an
infinite dimensional searching space.

When the system state is $n$ and the scheduling action $\bm m= d(n)$
is adopted, a \emph{holding cost} $h(n)$ and an \emph{operating
cost} $o(\bm m)$ will be incurred per unit of time. In the
literature, it is commonly assumed that the operating cost is
increasing with respect to (w.r.t.) the number of working servers.
In this paper, we define the linear operating cost function $o(\bm
m)$ as follows.
\begin{equation}
o(\bm m) := \sum_{k=1}^{K}m_k c_k = \bm m \bm c,
\end{equation}
where $\bm c:=(c_1,c_2,\cdots,c_K)^T$ is a $K$-dimensional column
vector and $c_k$ represents the operating cost rate per server in
group $k$. Therefore, the total cost rate function of the whole system
per time unit is defined as
\begin{equation}\label{eq_f}
f(n,\bm m) := h(n) + \bm m \bm c.
\end{equation}
We make the following assumption regarding the customer's holding
cost (waiting cost) and the server's setup cost (changeover cost).
\begin{assumption}\label{assumption1}
$h(n)$ is an increasing convex function w.r.t. $n$ and
$h(n)\rightarrow \infty$ when $n \rightarrow \infty$. The server's
setup cost is negligible.
\end{assumption}

Such a holding cost assumption is widely used in the literature
\cite{VanMieghem95} and represents the situation where the delay
cost grows rapidly as the system becomes more congested. For a
non-empty state $n$, if a scheduling action $\bm m$ is adopted, some
working servers may be turned off and services of customers at those
servers will be interrupted. These customers will be returned to the
waiting room or reassigned to other currently available working
servers. Such a rule is called \emph{non-resume transfer
discipline}. Since the setup cost for turning on a server (including
transferring a customer to an available server) is zero, we do not
have to keep track of the number of on (or off) servers for any
state. Otherwise, each server's status must be included in the
definition of the system state so that the state space will be
changed from one dimensional to multi-dimensional one, which is much
more complex.

Denote by $n_t$ the number of customers in the system at time $t
\geq 0$. The long-run average cost of the group-server queue under
policy $d$ can be written as
\begin{equation}\label{eq_eta}
\eta^d :=  \lim\limits_{T \rightarrow \infty}\mathbb E \left\{
\frac{1}{T} \int_{0}^{T} f(n_t,d(n_t))\dif t \right\}.
\end{equation}
The objective is to find the optimal policy $d^*$ such that the
associated long-run average cost is minimized. That is,
\begin{equation}\label{eq_prob}
d^* = \argmin \limits_{d \in \mathcal D} \{ \eta^d \}.
\end{equation}

\noindent\textbf{Remark 1.} It is worth noting that the scheduling
policy $d$ is a mapping from an infinite state space to a
$K$-dimensional finite action space. The state space is infinite and
the action space grows exponentially with $K$. Thus, the policy
space $\mathcal D$ to be searched is of infinite dimension.
Characterizing the optimal structure of such a mapping is
challenging but necessary in solving this optimization problem. A
major contribution of this paper is to accomplish this challenging
task and derive a simple $c/\mu$-rule as the optimal policy under a
certain condition.

\section{Optimal Policy Structure}\label{section_result}
The optimization problem (\ref{eq_prob}) can be modeled as a
continuous-time MDP with the long-run average cost criterion. The
traditional theory of MDPs is based on the well-known Bellman
optimality equation. However, in a multi-server queueing model with
infinite buffer, it may be difficult to characterize the structure
of the optimal policy using the traditional approach. Recently, Cao
proposed the sensitivity-based optimization (SBO) theory
\cite{Cao07}. This relatively new theory provides a new perspective
to optimize the performance of Markov systems. The key idea of the
SBO theory is to utilize the performance sensitivity information,
such as the performance difference or the performance derivative, to
conduct the optimization of stochastic systems. It may even treat
the stochastic optimization problems to which the dynamic
programming fails to offer the solution \cite{Cao07,Xia13,Xia15}. We
use the SBO theory to characterize the structure of the optimal
policy of the optimization problem (\ref{eq_prob}).

First, we study the structure of the action space. Owing to zero
setup cost, we should turn off any idle servers and obtain the
following result immediately.
\begin{proposition}\label{pro1}
The optimal action $\bm m$ at state $n$ satisfies $\bm m \bm 1\leq
n$, where $\bm 1$ is a column vector with proper dimension and its
all elements are 1's.
\end{proposition}

Note that if the server setup cost is not zero, this proposition may
not hold. From Proposition~\ref{pro1}, for every state $n$, we can
define the \emph{efficient action space} $\mathbb M_n$ as
\begin{equation}
\mathbb M_n := \{ \mbox{all } \bm m \in \mathbb M : \bm m \bm 1 \leq
n \}.
\end{equation}
A policy $d$ is said to be \emph{efficient} if $d(n) \in\mathbb M_n$
for every $n \in \mathbb N$. Accordingly, the \emph{efficient policy
space} $\mathcal D_e$ is defined as
\begin{equation}
\mathcal D_e := \{\mbox{all } d : d(n)\in \mathbb M_n, \forall n \in
\mathbb N\}.
\end{equation}
Therefore, in the rest of the paper, we limit our optimal policy
search in $\mathcal D_e$. For any efficient action $\bm m \in
\mathbb M_n$, the total service rate of the queueing system is $\bm
m \bm \mu$, where $\bm \mu$ is a $K$-dimensional column vector of
service rates defined as
\begin{equation}
\bm \mu := (\mu_1, \mu_2, \cdots, \mu_K)^T.
\end{equation}

For the continuous-time MDP formulated in (\ref{eq_prob}), we define
the \emph{performance potential} as follows \cite{Cao07}.
\begin{equation}\label{eq_g}
g(n) := \mathbb E\left\{ \int_{0}^{\infty} [f(n_t,d(n_t)) - \eta]
\dif t \Big | n_0 = n \right\}, \quad n \in \mathbb N,
\end{equation}
where $\eta$ is defined in (\ref{eq_eta}) and we omit the
superscript `$d$' for simplicity. The definition (\ref{eq_g})
indicates that $g(n)$ quantifies the long-run accumulated effect of
the initial state $n$ on the average performance $\eta$. In the
traditional MDP theory, $g(n)$ can also be understood as the
\emph{relative value function} or \emph{bias} \cite{Puterman94}.

By using the strong Markov property, we can decompose the
right-hand-side of (\ref{eq_g}) into two parts as follows.
\begin{eqnarray}\label{eq10}
g(n) &=& \mathbb E\{\tau\} [f(n,d(n)) - \eta] + \mathbb E\left\{
\int_{\tau}^{\infty} [f(n_t,d(n_t)) - \eta] \dif t \Big | n_0 = n
\right\} \nonumber\\
&=& \frac{1}{\lambda+d(n)\bm \mu} [f(n,d(n)) - \eta] +
\frac{\lambda}{\lambda+d(n)\bm \mu}\mathbb E\left\{
\int_{\tau}^{\infty} [f(n_t,d(n_t)) - \eta] \dif t \Big | n_{\tau} =
n+1 \right\} \nonumber\\
&& + \frac{d(n)\bm \mu}{\lambda+d(n)\bm \mu}\mathbb E\left\{
\int_{\tau}^{\infty} [f(n_t,d(n_t)) - \eta] \dif t \Big | n_{\tau} =
n-1 \right\},
\end{eqnarray}
where $\tau$ is the sojourn time at the current state $n$ and
$\mathbb E\{\tau\} =\frac{1}{\lambda+d(n)\bm \mu}$.

Combining (\ref{eq_g}) and (\ref{eq10}), we have the recursion
\begin{equation}\label{eq11}
\begin{array}{ll}
\left[\lambda+d(n)\bm \mu \right] g(n) = f(n,d(n)) - \eta + \lambda g(n+1) + d(n)\bm \mu g(n-1), & \quad n \geq 1; \\
\lambda g(n) = f(n,d(n)) - \eta + \lambda g(n+1), & \quad n = 0. \\
\end{array}
\end{equation}

We denote $\bm B$ as the infinitesimal generator of the Markov
process under an efficient policy $d \in \mathcal D_e$. Due to
nature of the continuous-time Markov process, the elements of $\bm
B$ are: for a state $n \geq 1$, $B(n,n) = -\lambda - d(n)\bm \mu$,
$B(n,n+1) = \lambda$, $B(n,n-1) = d(n)\bm \mu$, $B(n,:) = 0$
otherwise. Therefore, with such a birth-death process, $\bm B$ can
be written as the following form
\begin{equation}\label{eq_B}
\bm B = \left[
\begin{array}{cccccc}
-\lambda & \lambda & 0 & 0 & 0 & \cdots \\
d(1)\bm \mu & -\lambda-d(1)\bm \mu & \lambda & 0 & 0 & \cdots\\
0 & d(2)\bm \mu & -\lambda-d(2)\bm \mu & \lambda & 0 & \cdots\\
0 & 0 & d(3)\bm \mu & -\lambda-d(3)\bm \mu & \lambda & \cdots\\
\vdots & \vdots & \vdots & \ddots & \ddots & \ddots \\
\end{array}
\right].
\end{equation}
Hence, we can rewrite (\ref{eq11}) as follows.
\begin{equation}\label{eq12}
\begin{array}{ll}
-B(n,n) g(n) = f(n,d(n)) - \eta + B(n,n+1) g(n+1) + B(n,n-1) g(n-1), & n \geq 1; \\
-B(n,n) g(n) = f(n,d(n)) - \eta + B(n,n+1) g(n+1), & n = 0. \\
\end{array}
\end{equation}
We further denote $\bm g$ and $\bm f$ as the column vectors whose
elements are $g(n)$'s and $f(n,d(n))$'s, respectively. We can
rewrite (\ref{eq12}) in a matrix form as below.
\begin{equation}\label{eq_poisson}
\bm f - \eta \bm 1 + \bm B \bm g = \bm 0.
\end{equation}
The above equation is also called \emph{the Poisson equation} for
continuous-time MDPs with long-run average criterion \cite{Cao07}.
As $\bm g$ is called performance potential or relative value
function, we can set $g(0)=\zeta$ and recursively solve $g(n)$ based
on (\ref{eq12}), where $\zeta$ is any real number. Using matrix
operations, we can also evaluate $\bm g$ by solving the infinite
dimensional Poisson equation (\ref{eq_poisson}) through numerical
computation techniques, such as RG-factorizations \cite{Li04}.

For the stability of the queueing system, we impose a \emph{sufficient
condition} as follows.
\begin{proposition}\label{pro2}
If there exists a constant $\tilde{n}$ and for any $n \geq
\tilde{n}$, we always have $d(n) \bm \mu > \lambda$, then this
group-server queue under policy $d$ is stable and its steady state
distribution $\bm \pi$ exists.
\end{proposition}
Proposition~\ref{pro2} ensures that $\bm \pi$ exists under a proper
selection of policy $d$. Thus, we have
\begin{equation}
\begin{array}{l}
\bm \pi \bm B = \bm 0, \\
\bm \pi \bm 1 = 1.  \\
\end{array}
\end{equation}
The long-run average cost of the system can be written as
\begin{equation}
\eta = \bm \pi \bm f.
\end{equation}

Suppose the scheduling policy is changed from $d$ to $d'$, where $d,
d' \in \mathcal D_e$. Accordingly, all the associated quantities
under the new policy $d'$ will be denoted by $\bm B'$, $\bm f'$,
$\bm \pi'$, $\eta'$, etc. Obviously, we have $\bm \pi' \bm B' = \bm
0$, $\bm \pi' \bm 1 = 1$, and $\eta' = \bm \pi' \bm f'$.

Left-multiplying $\bm \pi'$ on both sides of (\ref{eq_poisson}), we
have
\begin{equation}
\bm \pi' \bm f - \eta \bm \pi' \bm 1 + \bm \pi' \bm B \bm g = 0.
\end{equation}
Using $\bm \pi' \bm B' = \bm 0$, $\bm \pi' \bm 1 =
1$, and $\eta' = \bm \pi' \bm f'$ , we can write (18) as
\begin{equation}
\eta' - \bm \pi' \bm f' + \bm \pi' \bm f - \eta + \bm \pi' \bm B \bm
g - \bm \pi' \bm B' \bm g = 0,
\end{equation}
which gives the \emph{performance difference formula} for the
continuous-time MDP as follows \cite{Cao07}.
\begin{center}
\begin{boxedminipage}{1\columnwidth}
\begin{equation}\label{eq_diff}
\eta' - \eta = \bm \pi' [(\bm B' - \bm B)\bm g + (\bm f' - \bm f)].
\end{equation}
\vspace{-13pt}
\end{boxedminipage}
\end{center}

Equation (\ref{eq_diff}) provides the sensitivity information about
the system performance, which can be used to achieve the
optimization. It clearly quantifies the performance change due to a
policy change. Although the exact value of $\bm \pi'$ may not be
known for every new policy $d'$, all its entries are always
nonnegative and even positive for those positive recurrent states.
Therefore, if we choose a proper new policy (with associated $\bm
B'$ and $\bm f'$) such that the elements of the column vector
represented by the square bracket in (\ref{eq_diff}) are always
nonpositive, then we have $\eta'-\eta \leq 0$ and the long-run
average cost of the system will be reduced. If there is at least one
negative element in the square bracket for a positive recurrent
state, then we have $\eta'-\eta < 0$ and the system average cost
will be reduced strictly. This is the main idea for policy
improvement based on the performance difference formula
(\ref{eq_diff}).

Using (\ref{eq_diff}), we examine the sensitivity of scheduling
policy on the long-run average cost of the group-server queue.
Suppose that we choose a new policy $d'$ which is the same as the
current policy $d$ except for the action at a particular state $n$.
For this state $n$, policy $d$ selects action $\bm m$ and policy
$d'$ selects action $\bm m'$, where $\bm m,\bm m' \in \mathbb M_n$.
Substituting (\ref{eq_f}) and (\ref{eq_B}) into (\ref{eq_diff}), we
have
\begin{eqnarray}\label{eq_diff2}
\eta' - \eta &=& \bm \pi' [(\bm B' - \bm B)\bm g + (\bm f' - \bm f)] \nonumber\\
&=& \pi'(n)[(\bm B'(n,:) - \bm B(n,:))\bm g + (f'(n) - f(n))]\nonumber\\
&=& \pi'(n)\left[\sum_{k=1}^{K}(m'_k-m_k)\mu_k (g(n-1) - g(n)) +
(\bm m' \bm c - \bm m \bm c)\right] \nonumber\\
&=& \pi'(n)\sum_{k=1}^{K}(m'_k-m_k)\left[c_k - \mu_k (g(n) -
g(n-1))\right],
\end{eqnarray}
where $g(n)$ is the performance potential of the system under the
current policy $d$. The value of $g(n)$ can be numerically computed
based on (\ref{eq_poisson}) or online estimated based on
(\ref{eq_g}). Details can be found in Chapter 3 of \cite{Cao07}.

For the purpose of analysis, we define a new quantity $G(n)$ as
below.
\begin{equation}\label{eq_G}
G(n) := g(n) - g(n-1), \quad n=1,2,\cdots.
\end{equation}
Note that $G(n)$ quantifies the performance potential difference
between neighboring states $n$ and $n-1$. According to the theory of
perturbation analysis (PA) \cite{Cao94,Ho91}, $G(n)$ is called
\emph{perturbation realization factor} (PRF) which measures the
effect on the average performance when the initial state is
perturbed from $n-1$ to $n$. For our specific problem
(\ref{eq_prob}), in certain sense, $G(n)$ can be considered as the
benefit of reducing the long-run average holding cost due to
operating a server. In the following analysis, $G(n)$ plays a
fundamental role of directly determining the optimal scheduling policy
for the group-server queue.

Based on the recursive relation of $\bm g$ in (\ref{eq11}), we can
also develop the following recursions for computing $G(n)$'s.
\begin{lemma}
The PRF $G(n)$ can be computed by the following recursive relations
\begin{equation}\label{eq_G-Recur}
\begin{array}{l}
G(n+1) = \frac{d(n)\bm \mu}{\lambda} G(n) + \frac{\eta -
f(n,d(n))}{\lambda}, \quad n \geq 1,\\
G(1) = \frac{\eta - f(0,d(0))}{\lambda}.\\
\end{array}
\end{equation}
\end{lemma}
\begin{proof}
From the second equation in (\ref{eq11}), we have
\begin{equation}
G(1) = g(1) - g(0) = \frac{\eta - f(0,d(0))}{\lambda}.
\end{equation}
Using the first equation in (\ref{eq11}), we have
\begin{equation}
\lambda(g(n+1) - g(n)) = d(n)\bm \mu (g(n) - g(n-1)) + \eta -
f(n,d(n)), \quad n \geq 1.
\end{equation}
Substituting (\ref{eq_G}) into the above equation, we directly have
\begin{equation}
G(n+1) = \frac{d(n)\bm \mu}{\lambda} G(n) + \frac{\eta -
f(n,d(n))}{\lambda}, \quad n \geq 1.
\end{equation}
Thus, the recursions for $G(n)$ are proved.
\end{proof}

Substituting (\ref{eq_G}) into (\ref{eq_diff2}), we obtain the
following performance difference formula in terms of $G(n)$ when the scheduling action
at a single state $n$ is changed from $\bm m$ to $\bm m'$.
\begin{center}
\begin{boxedminipage}{1\columnwidth}
\begin{equation}\label{eq_diff3}
\eta' - \eta = \pi'(n)\sum_{k=1}^{K}(m'_k-m_k) \left(c_k - \mu_k
G(n)\right).
\end{equation}
\vspace{-13pt}
\end{boxedminipage}
\end{center}

This difference formula can be extended to a general case when $d$
is changed to $d'$, i.e. $d(n)$ is changed to $d'(n)$ for all $n \in
\mathbb N$. Substituting the associated $(\bm B, \bm f)$ and $(\bm
B',\bm f')$ into (\ref{eq_diff}) yields
\begin{center}
\begin{boxedminipage}{1\columnwidth}
\begin{equation}\label{eq_diff4}
\eta' - \eta = \sum_{n \in \mathbb
N}\pi'(n)\sum_{k=1}^{K}(d'(n,k)-d(n,k)) \left(c_k - \mu_k
G(n)\right).
\end{equation}
\vspace{-13pt}
\end{boxedminipage}
\end{center}

Based on (\ref{eq_diff4}), we can directly obtain a condition for generating an improved policy as follows.
\begin{theorem}\label{theorem1}
If a new policy $d' \in \mathcal D_e$ satisfies
\begin{equation}\label{eq25}
(d'(n,k) - d(n,k))\left({c_k} - {\mu_k}G(n)\right) \leq 0
\end{equation}
for all $k=1,2,\cdots,K$ and $n \in \mathbb N$, then $\eta' \leq
\eta$. Furthermore, if for at least one state-group pair $(n,k)$,
the inequality in (\ref{eq25}) strictly holds, then $\eta' < \eta$.
\end{theorem}
\begin{proof}
Since (\ref{eq25}) holds for every $n$ and $k$ and $\pi'(n)$ is
always positive for ergodic processes, it follows from
(\ref{eq_diff4}) that $\eta' - \eta \leq 0$. Thus, the first part of
the theorem is proved. The second part can be proved using a similar
argument.
\end{proof}

Theorem~\ref{theorem1} provides a way to generate improved policies
based on the current feasible policy. For the system under the
current policy $d$, we compute or estimate $G(n)$'s based on its
definition. For every state $n$ and server group $k$, if we find
$\frac{c_k}{\mu_k}
> G(n)$, then we choose a smaller $d'(n,k)$; if we find
$\frac{c_k}{\mu_k} < G(n)$, then we choose a larger $d'(n,k)$
satisfying the condition $d'(n)\bm 1 \leq n$, as stated by
Proposition~\ref{pro1}. Therefore, according to
Theorem~\ref{theorem1}, the new policy $d'$ obtained from this
procedure will perform better than the current policy $d$. This
procedure can be repeated to continually reduce the system average
cost.

Note that the condition above is only a sufficient one to generate
improved policies. Now, we establish a \emph{necessary and
sufficient condition} for the optimal scheduling policy as follows.
\begin{theorem}\label{theorem2}
A policy $d^*$ is optimal if and only if its element $d^*(n)$, i.e.,
$(d^*(n,1),\cdots,d^*(n,K))$, is the solution to the following
integer linear programs
\begin{equation}\label{eq_ilp}
\mbox{\hspace{-2cm}\emph{ILP Problem:}} \hspace{1cm}\left\{
\begin{array}{l}
\min\limits_{d(n,k)}\left\{ \sum_{k=1}^{K}d(n,k)(c_k - \mu_k G^*(n))
\right\} \\
\mbox{s.t.} \quad 0 \leq d(n,k) \leq M_k, \\
\hspace{1cm}\sum_{k=1}^{K}d(n,k) \leq n,
\end{array}
\right.
\end{equation}
for every state $n \in \mathbb N$, where $G^*(n)$ is the PRF defined
in (\ref{eq_G}) under policy $d^*$.
\end{theorem}

\begin{proof}
First, we prove the sufficient condition. Suppose $d^*(n)$ is the
solution to the ILP problem (\ref{eq_ilp}), $\forall n \in \mathbb
N$. For any other policy $d' \in \mathcal D_e$, we know that it must
satisfy the constraints in (\ref{eq_ilp}) and
\begin{equation}\label{eq27}
\sum_{k=1}^{K}d'(n,k)(c_k - \mu_k G^*(n)) \geq
\sum_{k=1}^{K}d^*(n,k)(c_k - \mu_k G^*(n)), \quad \forall n \in
\mathbb N,
\end{equation}
since $d^*(n)$ is the solution to (\ref{eq_ilp}). Substituting
(\ref{eq27}) into (\ref{eq_diff4}), we obtain
\begin{eqnarray}
\eta' - \eta^* = \sum_{n \in \mathbb
N}\pi'(n)\sum_{k=1}^{K}(d'(n,k)-d^*(n,k))(c_k - \mu_k G^*(n)) \geq
0,
\end{eqnarray}
for any $d' \in \mathcal D_e$. Therefore, $\eta^*$ is the minimal
average cost of the scheduling problem (\ref{eq_prob}) and $d^*$ is
the optimal policy. The sufficient condition is proved.

Second, we use contradiction to prove the necessary condition.
Assume that the optimal policy $d^*$ is not always the solution to
the ILP problem (\ref{eq_ilp}). That is, at least for a particular
state $n$, there exists another $d(n)$ which is the solution to
(\ref{eq_ilp}) and satisfies
\begin{equation}\label{eq28}
\sum_{k=1}^{K}d(n,k)(c_k - \mu_k G^*(n)) <
\sum_{k=1}^{K}d^*(n,k)(c_k - \mu_k G^*(n)).
\end{equation}
Therefore, we can construct a new policy $d'$ as follows: It chooses
the action $d(n)$ at the state $n$ only and chooses the same actions
prescribed by $d^*$ at other states. Substituting $d'$ and $d^*$
into (\ref{eq_diff4}) gives
\begin{eqnarray}
\eta' - \eta^* = \pi'(n)\sum_{k=1}^{K}(d(n,k)-d^*(n,k))(c_k - \mu_k
G^*(n)).
\end{eqnarray}
Substituting (\ref{eq28}) into the above equation and using the fact
$\pi'(n)>0$ for any positive recurrent state $n$, we have $\eta' <
\eta^*$, which contradicts the assumption that $d^*$ is the optimal
policy. Thus, the assumption does not hold and $d^*$ should be the
solution to (\ref{eq_ilp}). The necessary condition is proved.
\end{proof}

Theorem~\ref{theorem2} indicates that the original scheduling
problem (\ref{eq_prob}) can be converted into a series of ILP
problem (\ref{eq_ilp}) at every state $n \in \mathbb N$. However, it
is impossible to directly solve an \emph{infinite} number of ILP
problems since the state space is infinite. To get around this
difficulty, we further investigate the structure of the solution to
these ILPs.

By analyzing (\ref{eq_ilp}), we can find that the solution to the
ILP problem must have the following structure:
\begin{itemize}
\item For those groups with $c_k - \mu_k G^*(n) > 0$, we have $d^*(n,k) = 0$;
\item For those groups with $c_k - \mu_k G^*(n) < 0$, we repeat letting $d^*(n,k) = M_k$ or as large as possible in
an ascending order of $c_k - \mu_k G^*(n)$, under the constraint
$\sum_{k=1}^{K}d(n,k) \leq n$.
\end{itemize}
Then, we can further specify the above necessary and sufficient
condition of the optimal policy as follows.
\begin{theorem}\label{theorem3}
A policy $d^*$ is optimal if and only if its element $d^*(n)$
satisfies the condition: If $G^*(n) > \frac{c_k}{\mu_k}$, then
$d^*(n,k) = M_k \wedge (n-\sum_{l=1}^{k-1}d^*(n,l))$; otherwise,
$d^*(n,k) = 0$, for $k=1,2,\cdots,K$, where the index of server
groups should be renumbered in an ascending order of $c_k-\mu_k
G^*(n)$ at the current state $n$, $n \in \mathbb N$.
\end{theorem}

With Theorem~\ref{theorem3}, we can see that the optimal policy can
be fully determined by the value of $c_k-\mu_k G^*(n)$. Such a policy form
is called an \emph{index policy} and $c_k - \mu_k G^*(n)$
can be viewed as an index, which has similarity to the
\emph{Gittins' index} or \emph{Whittle's index} for solving
multi-armed bandit problems \cite{Gittins11,Whittle88}.

Theorem~\ref{theorem3} also reveals the \emph{quasi bang-bang
control} structure of the optimal policy $d^*$. That is, the optimal
number of working servers in group $k$ is either 0 or $M_k$, except
for the group that first violates the efficient condition in
Proposition~\ref{pro1}. For any state $n$, after the group index is
renumbered according to Theorem 3, the optimal action always has the
following form
\begin{equation}\label{eq34}
d^*(n) = (M_1,M_2,\cdots,M_{\hat k-1},M_{\hat k} \wedge
(n-\sum_{l=1}^{\hat k-1}M_l), 0,0,\cdots,0),
\end{equation}
where $\hat k$ is the first group index that violates the constraint
$\sum_{l=1}^{\hat k}M_l \leq n$ or $c_{\hat k+1} - \mu_{\hat k+1}
G^*(n) < 0$, i.e.,
\begin{equation}\label{eq_hatK}
\hat{k} := \min \left\{k: \sum_{l=1}^{k}M_l > n, \mbox{ or }
\frac{c_{k+1}} {\mu_{k+1}} \geq G^*(n) \right\}.
\end{equation}
Therefore, $\hat k$ can also be viewed as a \emph{threshold} and we
have $\hat k \in \{0,1,\cdots,K\}$. Such a policy can be called a
\emph{quasi threshold policy} with threshold $\hat k$. Under this
policy, the number of servers to be turned on for each group is as
follows.
\begin{equation}
\left\{
\begin{array}{ll}
d^*(n,l) = M_l, \quad &\mbox{if }l<\hat k;\\
d^*(n,l) = 0, \quad &\mbox{if }l>\hat k;\\
d^*(n,l) = M_{\hat k} \wedge (n-\sum_{l=1}^{\hat k-1}M_l), \quad &\mbox{if }l=\hat k.\\
\end{array}
\right.
\end{equation}
If threshold $\hat k$ is determined, $d^*(n)$ is also determined.
Thus, finding $d^*(n)$ becomes finding the threshold $\hat k$, which
simplifies the search for the optimal policy. However, we note that
the index order of groups is renumbered according to the ascending
value of $c_k - \mu_k G^*(n)$, which is varied at different state
$n$ or different value of $G^*(n)$. On the other hand, the value of
$\hat k$ also depends on the system state $n$, $n \in \mathbb N$.
Therefore, the index order of groups and the threshold $\hat k$ will
both vary at different state $n$, which makes the quasi threshold
policy not easy to implement in practice. To further characterize
the optimal policy, we explore its other structural properties.

Difference formula (\ref{eq_diff4}) and Theorem~\ref{theorem3}
indicate that $c_k - \mu_k G(n)$ is an important quantity to
differentiate the server groups. If $c_k - \mu_k G(n) < 0$, turning
on servers in group $k$ can reduce the system average cost. Group
$k$ can be called an \emph{economic group} for the current system.
Therefore, we define $\mathbb K_n$ as the economic group set at the
current state $n$
\begin{equation}\label{eq_K}
\mathbb K_n := \left\{ k : G(n) > \frac{c_k}{\mu_k} \right\}.
\end{equation}
We should turn on severs in the economic groups $\mathbb K_n$ as
many as possible, subject to $d(n) \bm 1 \leq n$. Note
that $G(n)$ reflects the reduction of the holding cost due to
operating a server, from a long-run average perspective.

With Theorems~\ref{theorem2} and \ref{theorem3}, the optimization
problem (\ref{eq_prob}) for each state $n$ can be solved by finding
the solution to each subproblem in (\ref{eq_ilp}) with the structure
of quasi bang-bang control or quasi threshold form like
(\ref{eq34}). However, since the number of ILPs is infinite, we need
to establish the monotone property of PRF $G(n)$ which can convert
the infinite state space search to a finite state space search for
the optimal policy. To achieve this goal, we first establish the
convexity of performance potential $g^*(n)$.

\begin{theorem}\label{theorem4}
The performance potential $g^*(n)$ under the optimal policy $d^*$ is
increasing and convex in $n$.
\end{theorem}
\begin{proof} We prove this theorem by induction. Since the problem (\ref{eq_prob}) is a continuous time
MDP with the long-run average cost criterion, the optimal policy
$d^*$ should satisfy the \emph{Bellman optimality equation} as
follows.
\begin{equation}
\min\limits_{\bm m \in \mathbb M_n} \left\{ f(n,\bm m) - \eta^* +
\bm B(n,:|\bm m)\bm g^* \right\} = 0, \quad \forall n \in \mathbb N,
\end{equation}
where $\bm B(n,:|\bm m)$ is the $n$th row of the infinitesimal
generator $\bm B$ defined in (\ref{eq_B}) if action $\bm m$ is adopted.

Define $\Lambda$ as any constant that is larger than the maximal
absolute value of all elements in $\bm B$ under any possible policy.
Without loss of generality, we further define
\begin{equation}\label{eq_lambda}
\Lambda := \sup\limits_{n,\bm m}\{|B(n,n|\bm m)|\} = \lambda +
\sum_{k=1}^{K} M_k\mu_{k}.
\end{equation}
Then we can use the Bellman optimality equation to derive the
recursion for value iteration as follows.
\begin{eqnarray}\label{eq_gl+1}
\hspace{-0.6cm}\Lambda g_{l+1}(n) &=& \min\limits_{\bm m \in \mathbb
M_n}\left\{ f(n,\bm m) - \eta_{l} + \sum_{n' \in \mathbb N}
B(n,n'|\bm m)g_l(n') + \Lambda
g_l(n) \right\} \nonumber\\
&=& \min\limits_{\bm m \in \mathbb M_n}\left\{ h(n) + \bm m \bm c -
\eta_{l} + (\Lambda-\lambda-\bm m \bm \mu ) g_l(n) + \lambda
g_l(n+1) + \bm m \bm \mu g_l(n-1) \right\},
\end{eqnarray}
where the second equality holds because of using (\ref{eq_f}) and
(\ref{eq12}), $g_l(n)$ is the performance potential (relative value
function) of state $n$ at the $l$th iteration, and $\eta_l$ is the
long-run average cost at the $l$th iteration. By defining
\begin{equation}\label{eq_A}
A(n) := h(n) + \bm m \bm c - \eta_{l} + (\Lambda-\lambda-\bm m \bm
\mu ) g_l(n) + \lambda g_l(n+1) + \bm m \bm \mu g_l(n-1),
\end{equation}
we can rewrite (\ref{eq_gl+1}) as
\begin{equation}\label{eq_gl+2}
\Lambda g_{l+1}(n) = \min\limits_{\bm m \in \mathbb M_n}\left\{ A(n)
\right\}.
\end{equation}

It is well known from the MDP theory \cite{Puterman94} that the
initial value of $g_0$ can be any value. Therefore, we set $g_0(n) =
0$ for all $n$, which satisfies the increasing and convex property.
Now we use the induction to establish this property. Suppose
$g_l(n)$ is increasing and convex in $n$. We need to show that
$g_{l+1}(n)$ also has this property. If done, we know that $g_l(n)$
is increasing and convex in $n$ for all $l$. In addition, since the
value iteration converges to the optimal value function, i.e.,
\begin{equation}\label{eq_gl}
\lim\limits_{l \rightarrow \infty}g_l(n) = g^*(n), \quad n \in
\mathbb N,
\end{equation}
Then, we can conclude that $g^*(n)$ is increasing and convex in $n$.
The induction is completed in two steps.

First step, we prove the increasing property of $g_{l+1}(n)$ or
$g_{l+1}(n+1) - g_{l+1}(n) \geq 0$. Using (\ref{eq_gl+2}), we have
\begin{equation}\label{eq35}
\Lambda [g_{l+1}(n+1) - g_{l+1}(n)] = \min\limits_{\bm m \in \mathbb
M_{n+1}} \{ A(n+1) \} - \min\limits_{\bm m \in \mathbb
M_n}\left\{A(n) \right\}.
\end{equation}
Denote $\bm m^*_{n+1}$ as the optimal action in $\mathbb M_{n+1}$,
which achieves the minimum for $A(n+1)$ in (\ref{eq_gl+2}). Below,
we want to use $\bm m^*_{n+1}$ to remove the operators
$\min\limits_{\bm m \in \mathbb M_n}\{\cdot\}$ in (\ref{eq35}),
which has to be discussed in two cases by concerning whether $\bm
m^*_{n+1} \in \mathbb M_n$.

\noindent Case \textcircled{1}: $\bm m^*_{n+1} \in \mathbb M_n$, we
can directly use $\bm m^*_{n+1}$ to replace $\min\limits_{\bm m \in
\mathbb M_n}\{\cdot\}$ in (\ref{eq35}) and obtain
\begin{eqnarray}\label{eq36}
\Lambda [g_{l+1}(n+1) - g_{l+1}(n)] &\geq& A(n+1)|_{\bm m^*_{n+1}} -
A(n)|_{\bm
m^*_{n+1}}\nonumber\\
&&\hspace{-4cm}= h(n+1) + \bm m^*_{n+1} \bm c - \eta_{l} +
(\Lambda-\lambda-\bm m^*_{n+1}
\bm \mu ) g_l(n+1) + \lambda g_l(n+2) + \bm m^*_{n+1} \bm \mu g_l(n) \nonumber\\
&& \hspace{-3.6cm}  - [h(n) + \bm m^*_{n+1} \bm c - \eta_{l} +
(\Lambda-\lambda-\bm m^*_{n+1} \bm \mu ) g_l(n) + \lambda
g_l(n+1) + \bm m^*_{n+1} \bm \mu g_l(n-1) ] \nonumber\\
&&\hspace{-4cm}= [h(n+1) - h(n)] + (\Lambda-\lambda-\bm m^*_{n+1}
\bm \mu )
[g_l(n+1)-g_l(n)] \nonumber\\
&& \hspace{-3.6cm} + \lambda[g_l(n+2)-g_l(n+1)] + \bm m^*_{n+1} \bm
\mu [g_l(n)-g_l(n-1)].
\end{eqnarray}
From Assumption~\ref{assumption1} that $h(n)$ is increasing in $n$,
the first term of RHS of (\ref{eq36}) is non-negative. Moreover, we
already assume that $g_l(n)$ is increasing in $n$. We also know that
$(\Lambda-\lambda-\bm m^*_{n+1} \bm \mu ) \geq 0$ from the
definition (\ref{eq_lambda}). Therefore, with (\ref{eq36}), we have
$g_{l+1}(n+1) - g_{l+1}(n) \geq 0$ in this case.

\noindent Case \textcircled{2}: $\bm m^*_{n+1} \notin \mathbb M_n$,
it means that $\bm m^*_{n+1} \bm 1 = n+1 > n$ violating the
condition in Proposition~\ref{pro1}. In this case, we select an
action $\bm \alpha$ as below.
\begin{equation}
\bm \alpha = \bm m^*_{n+1} - \bm e_1,  \quad \bm \alpha \in \mathbb
M_n,
\end{equation}
where $\bm e_1$ is a zero vector except one proper element is 1 such
that every element of $\bm \alpha$ is nonnegative. We use $\bm
\alpha$ to replace $\min\limits_{\bm m \in \mathbb M_n}\{\cdot\}$ in
(\ref{eq35}) and obtain
\begin{eqnarray}\label{eq36b}
\Lambda [g_{l+1}(n+1) - g_{l+1}(n)] &\geq& A(n+1)|_{\bm m^*_{n+1}} -
A(n)|_{\bm \alpha}\nonumber\\
&&\hspace{-4cm}= h(n+1) + \bm m^*_{n+1} \bm c - \eta_{l} +
(\Lambda-\lambda-\bm m^*_{n+1}
\bm \mu ) g_l(n+1) + \lambda g_l(n+2) + \bm m^*_{n+1} \bm \mu g_l(n) \nonumber\\
&& \hspace{-3.6cm}  - [h(n) + \bm \alpha \bm c - \eta_{l} +
(\Lambda-\lambda-\bm \alpha \bm \mu ) g_l(n) + \lambda
g_l(n+1) + \bm \alpha \bm \mu g_l(n-1) ] \nonumber\\
&&\hspace{-4cm}= [h(n+1) - h(n)] + \bm e_1 \bm c +
\lambda[g_l(n+2)-g_l(n+1)] + (\Lambda-\lambda-\bm m^*_{n+1} \bm \mu
) g_l(n+1) \nonumber\\
&&\hspace{-3.6cm} + \bm m^*_{n+1} \bm \mu g_l(n) -
[(\Lambda-\lambda-\bm \alpha \bm \mu ) g_l(n) + \bm \alpha \bm \mu
g_l(n-1)] .
\end{eqnarray}
Since $g_l(n)$ is increasing in $n$, we have
\begin{eqnarray}
(\Lambda-\lambda-\bm m^*_{n+1} \bm \mu ) g_l(n+1) + \bm m^*_{n+1}
\bm \mu g_l(n) \geq (\Lambda-\lambda ) g_l(n); \\
- [(\Lambda-\lambda-\bm \alpha \bm \mu ) g_l(n) + \bm \alpha \bm \mu
g_l(n-1)] \geq -(\Lambda-\lambda ) g_l(n).
\end{eqnarray}
Substituting the above equations into (\ref{eq36b}), we have
\begin{equation}
\Lambda [g_{l+1}(n+1) - g_{l+1}(n)] \geq [h(n+1) - h(n)] + \bm e_1
\bm c + \lambda[g_l(n+2)-g_l(n+1)] > 0.
\end{equation}
Combining cases \textcircled{1}\&\textcircled{2}, we always have
$g_{l+1}(n+1) - g_{l+1}(n) \geq 0$ and the increasing property of
$g_{l+1}(n)$ is proved.

Second step, we prove the convex property of $g_{l+1}(n)$ or
$g_{l+1}(n+1) - 2g_{l+1}(n) + g_{l+1}(n-1) \geq 0$. We denote $\bm
m^*_{n-1}$ as the optimal action in $\mathbb M_{n-1}$, which
achieves the minimum for $A(n-1)$ in (\ref{eq_gl+2}). From
(\ref{eq_gl+2}), we have
\begin{eqnarray}\label{eq43a}
\Lambda [g_{l+1}(n+1) - 2g_{l+1}(n) + g_{l+1}(n-1)] &=&
\hspace{-0.5cm} \min\limits_{\bm m \in \mathbb M_{n+1}} \{ A(n+1) \}
- 2 \min\limits_{\bm m \in \mathbb M_{n}}\left\{A(n) \right\} +
\min\limits_{\bm m \in \mathbb M_{n-1}} \{ A(n-1)
\}\nonumber\\
&=& A(n+1)|_{\bm m^*_{n+1}} - 2 \min\limits_{\bm m \in \mathbb
M_n}\left\{A(n) \right\} + A(n-1)|_{\bm m^*_{n-1}}.
\end{eqnarray}
Similarly, we select actions to replace the operators
$\min\limits_{\bm m \in \mathbb M_n}\{\cdot\}$ in (\ref{eq43a}). The
actions are generated from $\bm m^*_{n+1}$ and $\bm m^*_{n-1}$ and
they should belong to the feasible set $\mathbb M_n$. We select two
actions $\bm \alpha_1$ and $\bm \alpha_2$ that satisfy
\begin{equation}\label{eq_alpha}
\bm \alpha_1 + \bm \alpha_2 = \bm m^*_{n+1} + \bm m^*_{n-1}, \quad
\bm \alpha_1, \bm \alpha_2 \in \mathbb M_n.
\end{equation}
For example, when $n$ is large enough and the condition in
Proposition~\ref{pro1} is always satisfied, we can simply select
$\bm \alpha_1 = \bm m^*_{n+1}$ and $\bm \alpha_2 = \bm m^*_{n-1}$.
For other cases where the condition in Proposition~\ref{pro1} may be
violated, we can properly adjust the number of working servers based
on $\bm m^*_{n+1}$ and $\bm m^*_{n-1}$ and always find feasible $\bm
\alpha_1$ and $\bm \alpha_2$. It is easy to verify and we omit the
details for simplicity. Therefore, we use $\bm \alpha_1$ and $\bm
\alpha_2$ to replace the operators $\min\limits_{\bm m \in \mathbb
M_n}\{\cdot\}$ of the two $A(n)$'s in (\ref{eq43a}) and obtain
\begin{equation}\label{eq43e}
\Lambda [g_{l+1}(n+1) - 2g_{l+1}(n) + g_{l+1}(n-1)] \geq
A(n+1)|_{\bm m^*_{n+1}} - A(n)|_{\bm \alpha_1} - A(n)|_{\bm
\alpha_2} + A(n-1)|_{\bm m^*_{n-1}}.
\end{equation}
Substituting (\ref{eq_A}) into the above equation, we have
\begin{eqnarray}
\Lambda [g_{l+1}(n+1) - 2g_{l+1}(n) + g_{l+1}(n-1)] &\geq&
[h(n+1)-2h(n)+h(n-1)] + [\bm m^*_{n+1}\bm c - (\bm \alpha_1+\bm
\alpha_2)\bm c \nonumber\\
&& \hspace{-8cm} + \bm m^*_{n-1}\bm c] + (\Lambda-\lambda-\bm
m^*_{n+1}\bm \mu)g_l(n+1) - (2\Lambda-2\lambda-\bm \alpha_1\bm \mu -
\bm \alpha_2\bm \mu) g_l(n)  + (\Lambda-\lambda-\bm m^*_{n-1}\bm \mu)g_l(n-1) \nonumber\\
&& \hspace{-8cm} + \lambda[g_l(n+2)-2g_l(n)+g_l(n-1)] + \bm
m^*_{n+1} \bm \mu g_l(n) - (\bm \alpha_1\bm \mu + \bm \alpha_2\bm
\mu) g_l(n-1) + \bm m^*_{n-1}\bm \mu g_l(n-2).
\end{eqnarray}
Since $h(n)$ is convex in $n$ (Assumption~\ref{assumption1}), we
have $h(n+1)-2h(n)+h(n-1) \geq 0$. Moreover, it is assumed that
$g_l(n)$ is convex in $n$, hence $g_l(n+2)-2g_l(n)+g_l(n-1) \geq 0$
holds. By further utilizing (\ref{eq_alpha}), we can derive
\begin{equation}
\bm m^*_{n+1}\bm c - (\bm \alpha_1+\bm \alpha_2)\bm c  + \bm
m^*_{n-1}\bm c = 0, \nonumber
\end{equation}
\begin{equation}
(\Lambda-\lambda-\bm m^*_{n+1}\bm \mu)g_l(n+1) -
(2\Lambda-2\lambda-\bm \alpha_1\bm \mu - \bm \alpha_2\bm \mu) g_l(n)
+ (\Lambda-\lambda-\bm m^*_{n-1}\bm \mu)g_l(n-1) \geq 0, \nonumber
\end{equation}
\begin{equation}
\bm m^*_{n+1} \bm \mu g_l(n) - (\bm \alpha_1\bm \mu + \bm
\alpha_2\bm \mu) g_l(n-1) + \bm m^*_{n-1}\bm \mu g_l(n-2) \geq 0.
\nonumber
\end{equation}
Therefore, we know $g_{l+1}(n+1) - 2g_{l+1}(n) + g_{l+1}(n-1) \geq
0$ and the convex property of $g_{l+1}(n)$ is proved.

In summary, we have proved that $g_l(n)$ is increasing and convex in
$n$ by induction, $\forall l \in \mathbb N$. Therefore, $g^*(n)$ is
also increasing and convex in $n$ by (\ref{eq_gl}). This completes
the proof.
\end{proof}

Since $G(n) = g(n) - g(n-1)$, from Theorem~\ref{theorem4}, we can
directly derive the following theorem about the monotone property of
$G^*(n)$.
\begin{theorem}\label{theorem5}
The PRF $G^*(n)$ under the optimal policy $d^*$ is nonnegative and
increasing in $n$.
\end{theorem}
Note that $G(n)$ plays a fundamental role in (\ref{eq_diff3}) and
(\ref{eq_diff4}). Thus, the increasing property of $G^*(n)$ enables
us to establish the monotone structure of optimal policy $d^*$ as
follows.
\begin{theorem}\label{theorem6_monotoned}
The optimal total number of working servers is increasing in $n$. In
other words, we have $||d^*(n+1)||_1 \geq ||d^*(n)||_1$, $\forall n
\in \mathbb N$.
\end{theorem}
\begin{proof}
Similar to (\ref{eq_K}), we define $\mathbb K^*_n$ as the set of
economic groups under the optimal policy $d^*$
\begin{equation}\label{eq_Ka}
\mathbb K^*_n := \left\{ k : G^*(n) > \frac{c_k}{\mu_k} \right\}.
\end{equation}
Note that $G^*(n+1) \geq G^*(n)$ from Theorem~\ref{theorem5} implies
that any $k \in \mathbb K^*_n$ also belongs to $\mathbb K^*_{n+1}$
or
\begin{equation}\label{eq44a}
\mathbb K^*_n \subseteq \mathbb K^*_{n+1}.
\end{equation}
Theorem~\ref{theorem3} indicates that the optimal total number of
working servers equals
\begin{equation}\label{eq45a}
||d^*(n)||_1 = \sum_{k \in \mathbb K^*_n} d^*(n,k) = \left\{
\begin{array}{ll}
\sum_{k \in \mathbb K^*_n} M_k, & \mbox{ if } n \geq \sum_{k \in \mathbb K^*_n} M_k;\\
n, & \mbox{ if } n < \sum_{k \in \mathbb K^*_n} M_k;\\
\end{array}
\right.
\end{equation}
which means
\begin{equation}\label{eq46a}
||d^*(n)||_1 = n \wedge \sum_{k \in \mathbb K^*_n} M_k.
\end{equation}
Therefore, for state $n+1$, we have
\begin{equation}\label{eq46b}
||d^*(n+1)||_1 = (n+1) \wedge \sum_{k \in \mathbb K^*_{n+1}} M_k.
\end{equation}
Utilizing (\ref{eq44a}) and comparing (\ref{eq46a}) and
(\ref{eq46b}), we directly obtain
\begin{equation}
||d^*(n+1)||_1 \geq ||d^*(n)||_1.
\end{equation}
This completes the proof.
\end{proof}

Theorem~\ref{theorem6_monotoned} rigorously confirms an intuitive
result that when the queue length increases, more servers should be
turned on to alleviate the system congestion, which is also the
essence of the congestion-based staffing policy \cite{Zhang09}.
However, it does not mean that the number of working servers in a
particular group is necessarily monotone increasing in $n$ (an
example is shown in Fig.~\ref{fig_ex1-b} in
Section~\ref{section_numerical}). More detailed discussion will be
given by Theorem~\ref{theorem8_monotoned} in the next section. Based
on Theorem~\ref{theorem6_monotoned}, we can further obtain the
following result.

\begin{corollary}\label{corollary1} For a state $\bar{n}$, if $d^*(\bar{n}, k) = M_k$, $\forall k$, then
$d^*(n, k) = M_k$, $\forall k$ and $n \geq \bar{n}$.
\end{corollary}

\noindent\textbf{Remark 2.} Corollary~\ref{corollary1} again
confirms an intuitive result that once the optimal action is turning
on all servers at certain state $\bar{n}$, then the same action is
optimal for all states larger than $\bar{n}$. Therefore, the search
for the optimal policy can be limited to the states $n < \bar{n}$
and the infinite state space is truncated without loss of
optimality. The difficulty of searching over the infinite state
space for the optimal policy can be avoided.

There exists a finite $\bar{n}$ (queue length) for which all servers
in all groups must be turned on. Such an existence of $\bar{n}$ can
be guaranteed by the linear operating cost and the increasing convex
holding cost function (Assumption 1), which can be verified by a
simple reasoning. Assume for any given scheduling policy, there is
no existence of such an $\bar{n}$, which means that there is at
least one idle server no matter how long the queue is (for all
states). With the increasing convex holding cost function, when the
queue length is long enough, the holding cost reduction due to the
work completion by turning on the idle server must exceed the
constant increase of server's operating cost. Then, turning on the
idle server at this state must reduce the system average cost.
Therefore, the assumed policy is never optimal. The suggested policy
of turning on the idle server must occur, which means the existence
of $\bar{n}$.

Based on Theorems~\ref{theorem1} and \ref{theorem3}, we can design a
procedure to find the optimal scheduling policy. First, compute the
value of $G(n)$ from Lemma 1 and determine the set $\mathbb K_n$ defined in
(\ref{eq_K}). If $n \geq \sum_{k \in \mathbb K_n} M_k$, turn on all
the servers in groups belonging to $\mathbb K_n$. If $n < \sum_{k
\in \mathbb K_n} M_k$, renumber the group indexes and set $d(n,k) =
M_k \wedge (n-\sum_{l=1}^{k-1}d(n,l))$, as stated in
Theorem~\ref{theorem3}. All the other servers should be off. This
process is repeated for $n=1,2,\cdots$, until the state $\bar{n}$
satisfying the condition in Corollary~\ref{corollary1} is reached. Set
$d(n,k)=M_k$ for all $k$ and $n \geq \bar{n}$ so that the whole
policy $d$ is determined. Then, we iterate the above procedure under
this new policy $d$. New improved policies will be repeatedly
generated until the policy cannot be improved and the procedure
stops. Based on this procedure, we develop the following
Algorithm~\ref{algo1} to find the optimal scheduling policy.

\begin{algorithm}[htbp]
  \caption{An iterative algorithm to find the optimal scheduling policy}\label{algo1}
  \begin{algorithmic}[1]

\State choose a proper initial policy $d^*$, e.g., $d^*(n,k)=M_k$,
$\forall n, k$, which indicates to turn on all servers;

\Repeat

\State set $d = d^*$, $d^* = \bm 0$, and $n=0$;

\State compute or estimate $\eta$ of the system under policy $d$;

\Repeat

    \State set $n = n+1$;

    \State compute $G(n)$ by using (\ref{eq_G-Recur}) recursively or by solving (\ref{eq_poisson}) and (\ref{eq_G});

    \State compute $\mathbb K_n$ using (\ref{eq_K});

    \If{$n \geq \sum_{k \in \mathbb K_n}M_k$}
        \State set $d^*(n,k) = M_k$, $\forall k \in \mathbb K_n$;
    \Else
        \State set $d^*(n,k) = M_k \wedge (n-\sum_{l=1}^{k-1}d^*(n,l))$, where $k \in \mathbb K_n$ and group indexes $k$'s are
renumbered according to ascending order of $c_k - \mu_k G(n)$, as
stated in Theorem~\ref{theorem3};

    \EndIf

\Until{$d^*(n,k) = M_k, \forall k$}

\State set $\bar{n} = n$;

\State set $d^*(n,k) = M_k$, $\forall n \geq \bar{n}, \forall k$;

\Until{$d = d^*$}

\Return  optimal $d^*$.

\end{algorithmic}
\end{algorithm}

From Algorithm~\ref{algo1}, we can see that this algorithm can
iteratively generate better policies. Such a manner is similar to
the policy iteration widely used in the traditional MDP theory. Note
that the index order of groups should be renumbered at every state
$n$, as stated in line 12 of Algorithm~\ref{algo1}. Since the server
groups are ranked by the index based on $c_k - \mu_k G^*(n)$, the
index sequence varies with state $n$. Moreover, although the total
number of working servers $||d^*(n)||_1$ is increasing in $n$,
$d^*(n,k)$ is not necessarily monotone increasing in $n$ for a
particular group $k$. This means that it is possible that for some
$n$ and $k$, we have $d^*(n,k) > d^*(n+1,k)$, as shown in
Fig.~\ref{fig_ex1-b} in Section~\ref{section_numerical}. These
complications may make it difficult to implement the optimal
scheduling policy in practical service systems with human servers,
as the servers have to be turned on or off without a regular
pattern. However, as we will discuss in the next section, the group
index sequence can remain unchanged if the ratio of cost rate to
service rate satisfies a reasonable condition. Then, we can develop
a simpler optimal scheduling policy obeying the $c/\mu$-rule, which
is much easier to implement in practice.

\section{The $c/{\mu}$-Rule}\label{section_rule}
We further study the optimal scheduling policy for the group-server
queue when the scale economies in terms of $c/{\mu}$ ratios exist.

\begin{assumption}\label{assumption2} (Scale Economies)
If the server groups are sorted in the order of $\frac{c_1}{\mu_1}
\leq \frac{c_2}{\mu_2} \leq \cdots \leq \frac{c_K}{\mu_K}$, then
their service rates satisfy $\mu_1 \geq \mu_2 \geq \cdots \geq
\mu_K$.
\end{assumption}

This assumption is reasonable in some practical situations as it
means that a faster server has a smaller operating cost rate per
unit of service rate. This can be explained by \emph{the effect of
the scale economies}. For example, in a data center, a faster
computer usually has a lower cost per unit of computing capacity.
With Assumption~\ref{assumption2}, we can verify that the group
index according to the ascending order of $c_k - \mu_k G^*(n)$
remains unvaried, no matter what the value of $G^*(n)$ is. The
ascending order of $c_k - \mu_k G^*(n)$ is always the same as the
ascending order of $c_k/\mu_k$. That is, the optimal policy
structure in Theorem~\ref{theorem3} can be characterized as follows.

\begin{theorem}\label{theorem7}
With Assumption~\ref{assumption2}, a policy $d^*$ is optimal if and
only if it satisfies the condition: If $G^*(n) > \frac{c_k}{\mu_k}$,
then $d^*(n,k) = M_k \wedge (n-\sum_{l=1}^{k-1}d^*(n,l))$;
otherwise, $d^*(n,k) = 0$, for $k=1,2,\cdots,K$, $n \in \mathbb N$.
\end{theorem}

Theorem~\ref{theorem7} implies that the optimal policy $d^*$ follows
a simple rule called \emph{the $c/{\mu}$-rule}: \emph{Servers in the
group with smaller $c/{\mu}$ ratio should be turned on with higher
priority}. This rule is very easy to implement as the group index
renumbering for each state in Theorem~\ref{theorem3} is not needed
anymore. As mentioned earlier, the $c/{\mu}$-rule can be viewed as a
counterpart of the famous \emph{$c \mu$-rule} for the scheduling of
polling queues \cite{Smith56,VanMieghem95}, in which the queue with
greater $c \mu$ will be given higher priority to be served by the
single service facility.

Using the monotone increasing property of $G^*(n)$ in
Theorem~\ref{theorem5} and Assumption~\ref{assumption2}, we can
further characterize the monotone structure of the optimal policy
$d^*$ as follows.

\begin{theorem}\label{theorem8_monotoned}
The optimal scheduling action $d^*(n,k)$ is increasing in $n$,
$\forall k=1,2,\cdots,K$.
\end{theorem}

\begin{proof}
First, it follows from Theorem~\ref{theorem5} that
\begin{equation}
G^*(n+1) \geq G^*(n), \qquad \forall n \in \mathbb N.
\end{equation}
From Theorem~\ref{theorem7}, we know that for any state $n$, if
$G^*(n) > \frac{c_k}{\mu_k}$, the optimal action is
\begin{equation}\label{eq44}
d^*(n,k) = M_k \wedge (n-\sum_{l=1}^{k-1}d^*(n,l)).
\end{equation}
Therefore, for state $n+1$, we have $G^*(n+1) \geq G^*(n) >
\frac{c_k}{\mu_k}$. Thus, the optimal action is
\begin{equation}\label{eq45}
d^*(n+1,k) = M_k \wedge (n+1 -\sum_{l=1}^{k-1}d^*(n+1,l)).
\end{equation}
Since $d^*(n)$ has a quasi threshold structure, there exists a
certain $\hat k$ defined in (\ref{eq_hatK}) such that $d^*(n)$ has
the following form
\begin{equation}
d^*(n) = (M_1,M_2,\cdots,M_{\hat k-1},M_{\hat k} \wedge
(n-\sum_{l=1}^{\hat k-1}M_l), 0,\cdots,0).
\end{equation}
Therefore, with (\ref{eq44}) and (\ref{eq45}) we have
\begin{equation}
d^*(n+1,l) = M_l = d^*(n,l), \qquad l=1,2,\cdots, \hat k-1.
\end{equation}
For group $\hat k$, we have
\begin{equation}\label{eq46}
d^*(n+1,\hat k) = M_{\hat k} \wedge (n+1 -\sum_{l=1}^{\hat k-1}M_l)
\ \geq \ d^*(n,\hat k) = M_{\hat k} \wedge (n -\sum_{l=1}^{\hat
k-1}M_l).
\end{equation}
More specifically, if $d^*(n,\hat k) < M_{\hat k}$, the inequality
in (\ref{eq46}) strictly holds and we have $d^*(n+1) =
(M_1,M_2,\cdots,$ $M_{\hat k-1},d^*(n,\hat k)+1, 0,\cdots,0)$. If
$d^*(n,\hat k) = M_{\hat k}$, the equality in (\ref{eq46}) holds and
we have $d^*(n+1) = (M_1,M_2,\cdots,M_{\hat k-1},M_{\hat k},
1,0,\cdots,0)$ for $\hat k < K$ or $d^*(n+1) =
(M_1,M_2,\cdots,M_{K-1},M_{K})$ for $\hat k=K$.

Therefore, we always have $d^*(n+1,k) \geq d^*(n,k)$, $\forall
k=1,2,\cdots,K$. This completes the proof.
\end{proof}

\noindent\textbf{Remark 3.} Theorem~\ref{theorem8_monotoned} implies
that $d^*(n)$ is increasing in $n$ in vector sense. That is
$d^*(n+1) \geq d^*(n)$ in vector comparison. Therefore, we certainly have
$||d^*(n+1)||_1 \geq ||d^*(n)||_1$, as indicated in
Theorem~\ref{theorem6_monotoned}.

Using the monotone property of $G^*(n)$ in Theorem~\ref{theorem5}
and the $c/{\mu}$-rule in Theorem~\ref{theorem7}, we can directly
obtain the multi-threshold policy as the optimal policy,
which can be viewed as a generalization of the two-threshold CBS
policy in our previous study \cite{Zhang09}.

\begin{theorem}\label{theorem9}
The optimal policy $d^*$ has a multi-threshold form with thresholds
$\theta_k$: If $n \geq \theta_{k}$, the maximum number of servers in
group $k$ should be turned on, $\forall n \in \mathbb N$,
$k=1,2,\cdots,K$.
\end{theorem}
\begin{proof}
We know that $G^*(n)$ increases with $n$ from
Theorem~\ref{theorem5}. For any particular group $k$, we can define
a threshold as
\begin{equation}
\theta_k := \min\left\{n: G^*(n)>\frac{c_k}{\mu_k}\right\}, \quad
k=1,2,\cdots,K.
\end{equation}
Therefore, for any $n \geq \theta_k$, $G^*(n)>\frac{c_k}{\mu_k}$ and
the servers in group $k$ should be turned on as many as possible,
according to the $c/\mu$-rule in Theorem~\ref{theorem7}. Thus, the
optimal scheduling for servers in group $k$ has a form of threshold
$\theta_k$ and the theorem is proved.
\end{proof}

Note that ``the maximum number of servers in group $k$ should be
turned on" in Theorem~\ref{theorem9} means that the optimal action
$d^*(n,k)$ should obey the constraint in Proposition~\ref{pro1},
i.e., $d^*(n,k) = M_k \wedge (n-\sum_{l=1}^{k-1}d^*(n,l))$.
Theorem~\ref{theorem9} implies that the policy of the original
problem (\ref{eq_prob}) can be represented by a $K$-dimensional
threshold vector $\bm \theta$ as below.
\begin{equation}
\bm \theta := (\theta_1,\theta_2,\cdots,\theta_K),
\end{equation}
where $\theta_k \in \mathbb N$. With the monotone property of
$G^*(n)$ and Theorem~\ref{theorem7}, we can directly derive that
$\theta_k$ is monotone in $k$, i.e.,
\begin{equation}
\theta_1 \leq \theta_2 \leq \cdots \leq \theta_K.
\end{equation}
As long as $\bm \theta$ is given, the associated policy $d$ can be
recursively determined as below.
\begin{equation}\label{eq_threshold_d}
\begin{array}{ll}
d(n,k) = M_k \wedge \Big(n - \sum_{l=1}^{k-1} d(n,l)\Big), & \mbox{if } n \geq \theta_k; \\
d(n,k) = 0, & \mbox{if } n < \theta_k; \\
\end{array}
\end{equation}
where $n \in \mathbb N$, $k=1,2,\cdots,K$.

We can further obtain the constant optimal
threshold for group 1.
\begin{theorem}\label{theorem_th1}
The optimal threshold of group 1 is always $\theta^*_1 = 1$, that is,
we should always utilize the most efficient server group whenever
any customer presents in the system.
\end{theorem}
\begin{proof} We use contradiction argument to prove this theorem. Assume
that the optimal threshold policy is $\bm \theta^*$ with $\theta^*_1
> 1$. We denote by $\bm X^{\bm \theta^*}=\{X^{\bm \theta^*}(t), t \geq
0\}$ the stochastic process of the queueing system under this policy
$\bm \theta^*$, where $X^{\bm \theta^*}(t)$ is the system state at
time $t$. We construct another threshold policy as $\tilde{\bm
\theta} = \bm \theta^* - (\theta^*_1 - 1)\bm 1^T$, where $\bm 1$ is
a $K$-dimensional column vector with 1's. Therefore, we know
$\tilde{\theta_1} = 1$. We denote by $\bm X^{\tilde{\bm \theta}}$
the stochastic process of the queueing system under policy
$\tilde{\bm \theta}$.

As $\theta^*_1 > 1$, we know that any state in the set
$\{0,1,\cdots,\theta^*_1-2\}$ is a \emph{transient state} of $\bm
X^{\bm \theta^*}$. Since transient states have no contribution to
the long-run average cost $\eta$, the statistics of $\bm
X^{\tilde{\bm \theta}}$ is equivalent to those of $\{X^{\bm
\theta^*}(t) - (\theta^*_1 - 1)\}$ if we omit the transient states.
Since the holding cost $h(n)$ is an increasing convex function in
$n$, it is easy to verify that $\eta^{\bm \theta^*} \geq
\eta^{\tilde{\bm \theta}}$. That is, if we simultaneously decrease
the thresholds of $\bm \theta^*$ to $\tilde{\bm \theta}$, the system
average cost will be decreased. Therefore, the assumption is not
true and $\theta^*_1 = 1$ is proved.
\end{proof}

Theorem~\ref{theorem9} indicates that the optimization problem
(\ref{eq_prob}) over an infinite state space is converted to the
problem of finding the optimal thresholds $\theta^*_k$, where
$k=1,2,\cdots,K$. Denoting by $\mathbb N^{K}_{\uparrow}$ a
$K$-dimensional positive integer space with its elements satisfying
$\theta_1 \leq \theta_2 \leq \cdots \leq \theta_K$, the original
problem (\ref{eq_prob}) can be rewritten as
\begin{equation}\label{eq_prob2}
\bm \theta^* = \argmin\limits_{\bm \theta \in \mathbb
N^{K}_{\uparrow}}\{\eta^{\bm \theta}\}.
\end{equation}
Therefore, the state-action mapping policy ($\mathbb N \rightarrow
\mathbb M$) is replaced by a parameterized policy with thresholds
$\bm \theta$. The original policy space is reduced from an infinite
dimensional space $\mathcal D$ to a $K$-dimensional integer space
$\mathbb N^{K}_{\uparrow}$. The \emph{curse of dimensionality} of
action space $\mathbb M$ and the optimal policy search over an
infinite state space $\mathbb N$ can be avoided by focusing on the
multi-threshold policies. To illustrate the procedure of policy
space reduction, we give an example of a 2-group server queue
illustrated in Fig.~\ref{fig_Policyredc}. We observe that the policy
space is significantly reduced after applying
Theorems~\ref{theorem3}, \ref{theorem7}, and \ref{theorem9}, which
identify the optimality structures. Since $\theta^*_1 = 1$ according
to Theorem~\ref{theorem_th1}, we only need to search for $K-1$
optimal thresholds. Sometimes we also treat $\theta^*_1$ as a
variable in order to maintain a unified presentation.

\begin{figure}[htbp]
\center
\includegraphics[width=1\columnwidth]{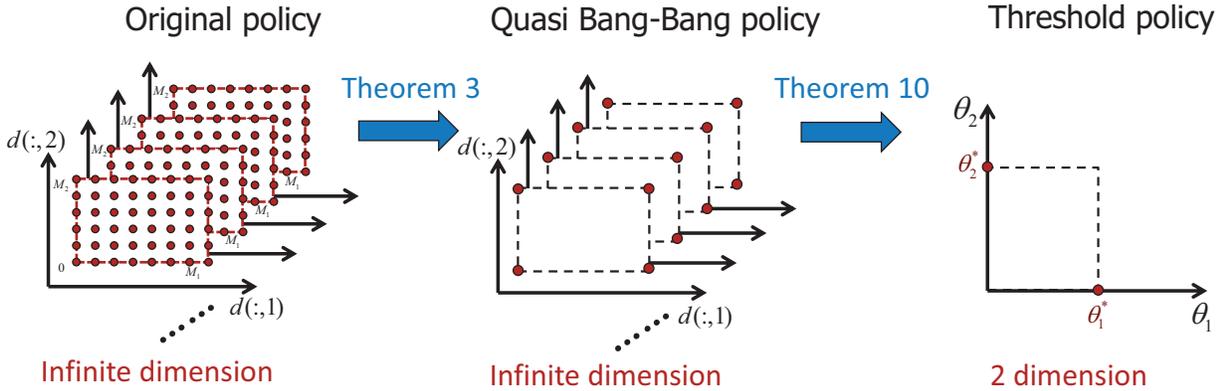}
\caption{Illustration of policy space reduction with an example of
2-group server queue.}\label{fig_Policyredc}
\end{figure}

By utilizing the $c/{\mu}$-rule and the optimality of
multi-threshold policy, we can further simplify
Algorithm~\ref{algo1} to find the optimal threshold policy $\bm
\theta^*$, which is described as Algorithm~\ref{algo2}. We can see
that Algorithm~\ref{algo2} iteratively updates the threshold policy
$\bm \theta$. Consider two threshold policies $\bm \theta$ and $\bm
\theta'$ generated from two successive iterations, respectively, by
using Algorithm~\ref{algo2}. $d'(n,k)$ and $d(n,k)$ are the
associated scheduling action determined by (\ref{eq_threshold_d})
based on $\bm \theta'$ and $\bm \theta$, respectively. From
(\ref{eq_threshold_d}), we see that $d'(n,k)$ and $d(n,k)$ have the
following relation.
\begin{equation}\label{eq_threshold_dd'}
\begin{array}{ll}
d'(n,k) = d(n,k) = M_k \wedge \Big(n - \sum_{l=1}^{k-1} d(n,l)\Big), & \mbox{if } n \geq \theta_k' \vee \theta_k; \\
d'(n,k) = d(n,k) = 0, & \mbox{if } n < \theta_k' \wedge \theta_k; \\
d'(n,k) = 0, d(n,k) = M_k \wedge \Big(n - \sum_{l=1}^{k-1} d(n,l)\Big), & \mbox{if } (\theta_k' \wedge \theta_k \leq n < \theta_k' \wedge \theta_k) \ \& \ (\theta_k'>\theta_k); \\
d(n,k) = 0, d'(n,k) = M_k \wedge \Big(n - \sum_{l=1}^{k-1} d'(n,l)\Big), & \mbox{if } (\theta_k' \wedge \theta_k \leq n < \theta_k' \wedge \theta_k) \ \& \ (\theta_k'<\theta_k). \\
\end{array}
\end{equation}
Substituting (\ref{eq_threshold_dd'}) into (\ref{eq_diff4}), we can
derive the following performance difference formula that quantifies
the effect of the change of threshold policy from $\bm \theta$ to
$\bm \theta'$, where $\bm \theta, \bm \theta' \in \mathbb
N^{K}_{\uparrow}$.

\begin{center}
\begin{boxedminipage}{1\columnwidth}
\begin{equation}\label{eq_diff5}
\eta' - \eta = \sum_{k=1}^{K}\sum_{n=\theta_k' \wedge
\theta_k}^{(\theta_k' \vee \theta_k)-1}\pi'(n)\left(d'(n,k) -
d(n,k)\right) \left(c_k - \mu_k G(n)\right),
\end{equation}
\vspace{-13pt}
\end{boxedminipage}
\end{center}
where $d'(n,k)$ and $d(n,k)$ are determined by
(\ref{eq_threshold_dd'}). From line 8-11 in Algorithm~\ref{algo2},
we observe that once $G(n)$ is larger than $\frac{c_k}{\mu_k}$, we
should set $\theta_k^* = n$ and turn on as many servers as possible
in group $k$. Groups with smaller $\frac{c_k}{\mu_k}$ will be turned
on with higher priority, which is the $c/\mu$-rule stated in
Theorem~\ref{theorem7}. With performance difference formula
(\ref{eq_diff5}), we see that the long-run average cost of the
system will be reduced after each policy update in
Algorithm~\ref{algo2}. When the algorithm stops, it means that the
system average cost cannot be reduced anymore and the optimal
threshold $\bm \theta^*$ is obtained. This procedure is also similar
to the policy iteration in the traditional MDP theory.

\begin{algorithm}[htbp]
  \caption{A $c/{\mu}$-rule based algorithm to find the optimal multi-threshold policy.}\label{algo2}
  \begin{algorithmic}[1]

\State renumber the groups index in an ascending order of their
value $\frac{c_k}{\mu_k}$, i.e., we have $\frac{c_1}{\mu_1} <
\frac{c_2}{\mu_2} < \cdots < \frac{c_K}{\mu_K}$;

\State choose the initial threshold as $\bm
\theta^*=(0,0,\cdots,0)_K$, which indicates to always turn on all
servers;

\Repeat

\State set $\bm \theta = \bm \theta^*$, $n=1$, and $k=1$;

\State compute or estimate $\eta$ of the system under threshold
policy $\bm \theta$;

\While{$k \leq K$ }

    \State compute $G(n)$ by using (\ref{eq_G-Recur}) recursively or by solving (\ref{eq_poisson}) and (\ref{eq_G});

    \While{($G(n) > \frac{c_k}{\mu_k}$) $\&$ ($k \leq K$) }

        \State set $\theta^*_k = n$;

        \State set $k = k+1$;

    \EndWhile

    \State set $n = n+1$;

\EndWhile

\Until{$\bm \theta = \bm \theta^*$}

\Return optimal  $\bm \theta^*$.

\end{algorithmic}
\end{algorithm}

Comparing Algorithms~\ref{algo1} and \ref{algo2}, we observe that
the essence of these two algorithms is similar: computing $G(n)$ and
updating policies iteratively. However, Algorithm~\ref{algo2} is
much simpler as it utilizes the $c/\mu$-rule based multi-threshold
policy. The $c/\mu$-rule, as an optimal policy, is very
easy to implement in practice. After the value of $G(n)$ is
obtained, we compare it with the groups' $\frac{c_k}{\mu_k}$ values.
If $\frac{c_k}{\mu_k}$ is smaller, we should turn on as many servers
as possible in group $k$; otherwise, turn off all servers in group
$k$. Such a procedure will induce a multi-threshold type policy, as
stated in Theorem~\ref{theorem9}.

More intuitively, we graphically demonstrate the above procedure by
using an example in Fig.~\ref{fig_CMuMonotone}. The vertical axis
represents the $c/{\mu}$ value of server groups, which are sorted in
an ascending order. When $n$ increases and the system becomes more
congested, we compute the value of the associated $G(n)$'s. As long
as $G(n)$ is larger than $\frac{c_k}{\mu_k}$, we should turn on as
many servers as possible for group $k$ and the associated $n$ is set
as the threshold $\theta_k$. For the case of $n=6$, group 2 still
has 1 server off although its $c/{\mu}$ is smaller than $G(n)$. It
is because of Proposition~\ref{pro1} that the total number of
working servers should not exceed $n$. Therefore, we can see that
the $c/\mu$-rule will prescribe to turn on group servers
from-bottom-up, as illustrated in Fig.~\ref{fig_CMuMonotone}. This
example demonstrates the monotone structure of the $c/{\mu}$-rule
and the optimal threshold policy.

Although Assumption 2 is reasonable for systems with non-human
servers such as computers with different performance efficiencies
(faster computers have smaller operating cost of processing each
job), the scale economies may not exist in systems with human
servers such as call centers where a faster server may incur much
higher operating cost. Thus, it is necessary to investigate the
robustness of the $c/\mu$-rule when Assumption 2 is not satisfied.
This is done numerically by Example~6 in the next section.

\begin{figure}[htbp]
\center
\includegraphics[width=1\columnwidth]{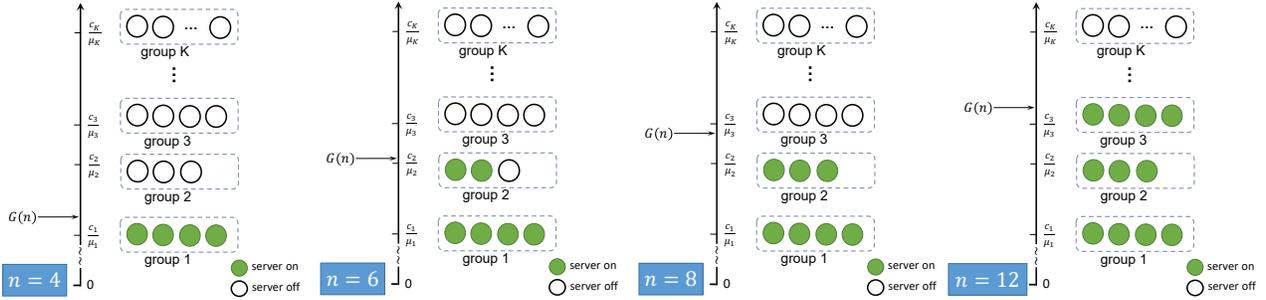}
\caption{An example of $c/{\mu}$-rule to determine servers' on-off,
where groups are sorted in ascending order of their $c/{\mu}$
values.}\label{fig_CMuMonotone}
\end{figure}

\section{Numerical Experiments}\label{section_numerical}
In this section, we conduct numerical experiments to verify the
analytical results and gain useful insights about optimal policies.

\subsection{Example 1: A general index policy case}
First, we consider a system with 3 groups of servers. System
parameters are as follows.

\begin{itemize}
\item Holding cost rate function: $h(n)=n$;
\item Arrival rate: $\lambda = 10$;
\item Number of groups: $K = 3$;
\item Number of servers in groups: $\bm M=(M_1,M_2,M_3)=(3,4,3)$;
\item Service rates of groups: $\bm \mu = (6,4,2)$;
\item Operating cost rates of groups: $\bm c = (7,4,3)$.
\end{itemize}

Note that Assumption~\ref{assumption1} is satisfied since the
holding cost rate function is $h(n)=n$ which is a linear function.
However, Assumption~\ref{assumption2} is not satisfied in this
example as the descending order of $\bm \mu$ is different from the
ascending order of ${c}/{\mu}$. Thus, the $c/{\mu}$-rule does not
apply to this example. We use Algorithm~\ref{algo1} to find the
optimal scheduling policy $d^*$ with the minimal average cost of
$\eta^*=12.5706$. The average queue length $L$ (including customers
in service) at each iteration is also illustrated along with the
long-run average cost $\eta$ in Fig.~\ref{fig_ex1-a-eta}. Since the
holding cost function is $h(n) = n$, the long-run average holding
cost is the same as $L$. Thus, the difference between $\eta$ and $L$
curves is the average operating cost. Note that $L$ significantly
increases at the second iteration, which corresponds to a scenario
with fewer servers working and more customers waiting. As shown in
Fig.~\ref{fig_ex1-a-eta}, the optimal solution is obtained after 4
iterations. We also plot the convex performance potential $g^*$ and
the increasing PRF $G^*$ under the optimal policy $d^*$ in
Fig.~\ref{fig_ex1-a-gG}, as predicted in Theorems 4 and 5.

\begin{figure}[htbp]
\centering \subfigure[Average cost and queue length during
iterations.]
{\includegraphics[width=0.49\columnwidth]{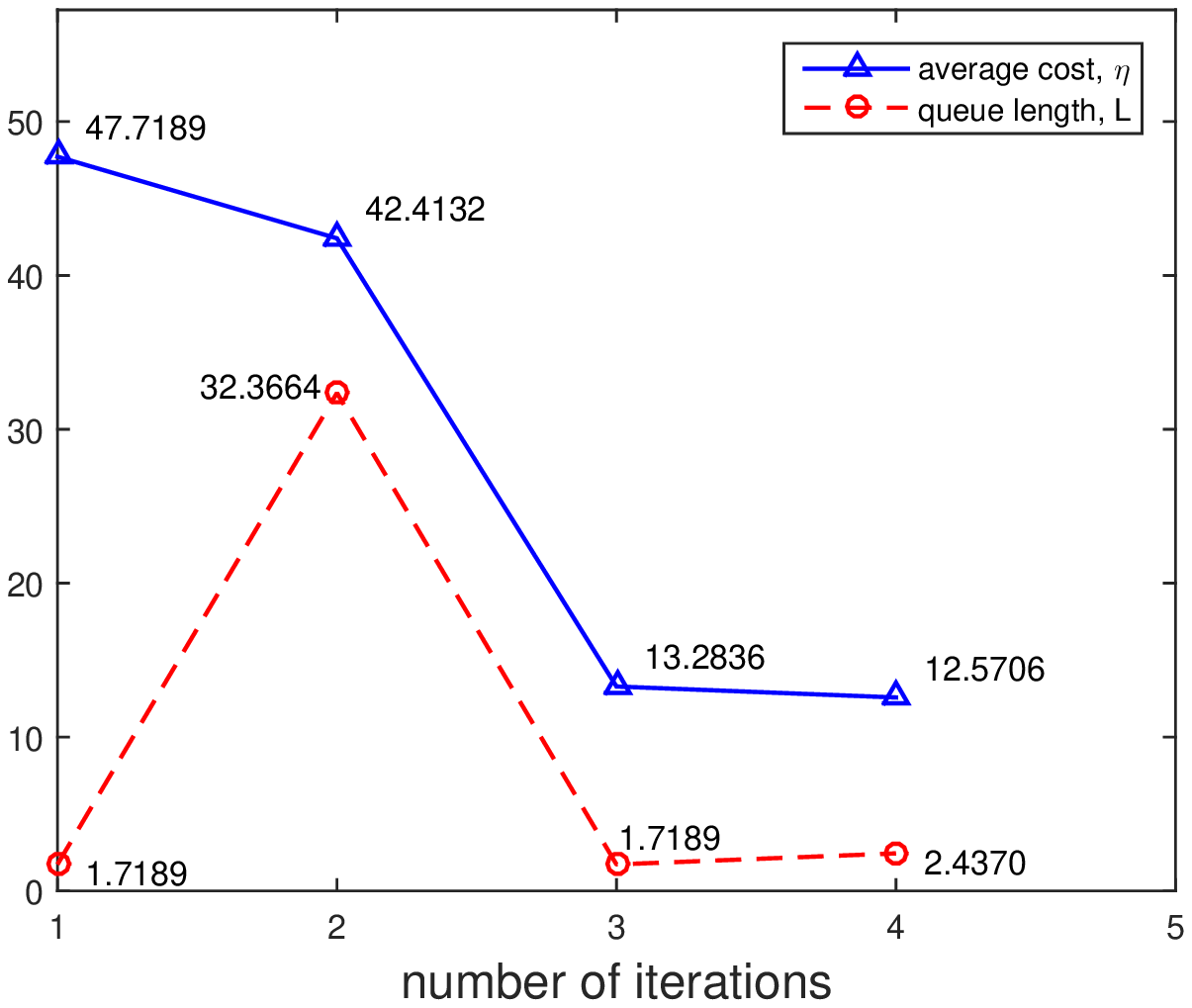}\label{fig_ex1-a-eta}}
\subfigure[Curves of $g^*$ and $G^*$.]
{\includegraphics[width=0.49\columnwidth]{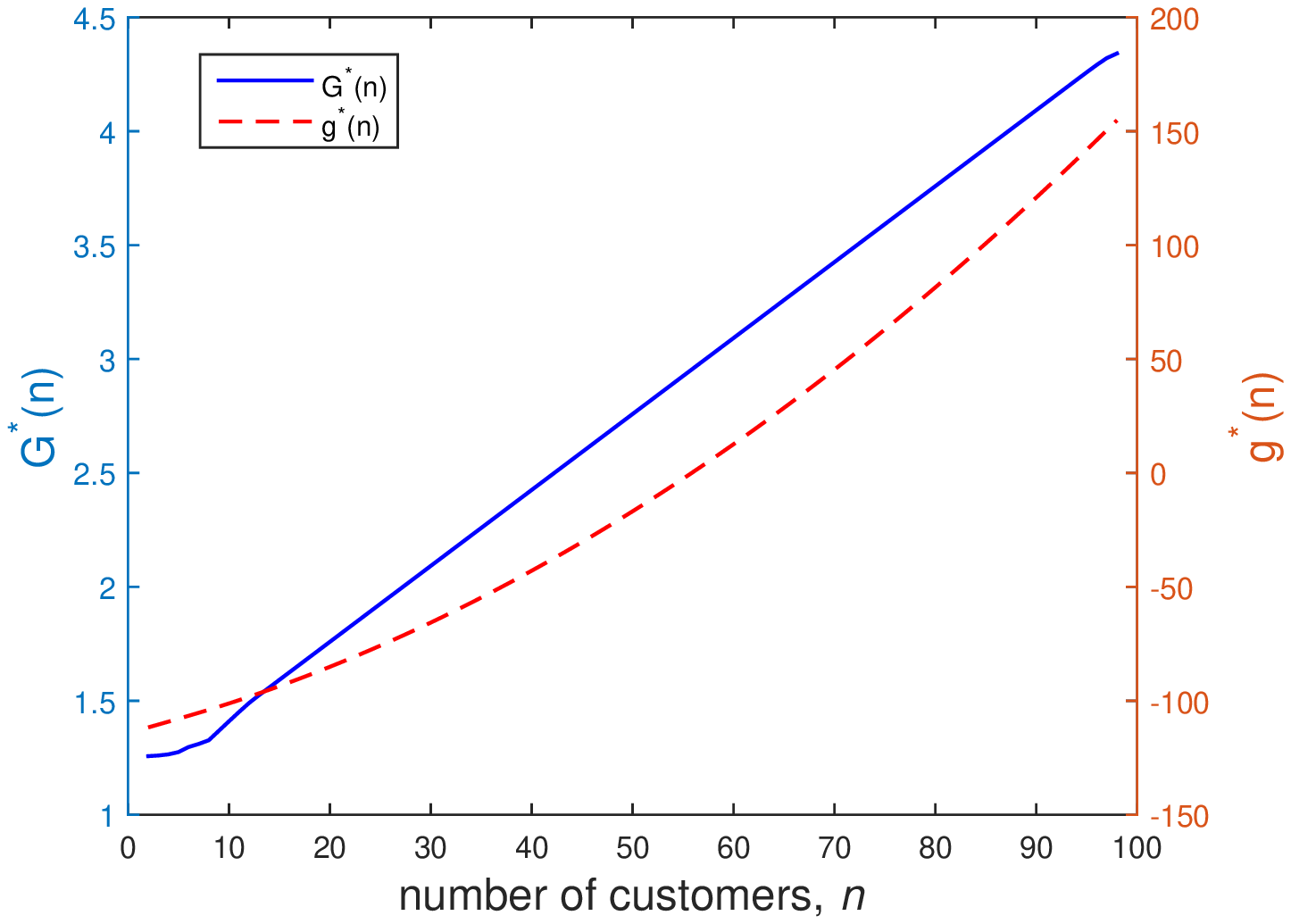}\label{fig_ex1-a-gG}}
\caption{Optimization procedure and curves of $g^*$ and $G^*$,
$\lambda=10$, $K=3$, $\bm M=(3,4,3)$, $\bm \mu=(6,4,2)$, $\bm
c=(7,4,3)$, $\eta^*=12.5706$.}\label{fig_ex1a}
\end{figure}

The optimal scheduling policy is shown in Fig.~\ref{fig_ex1-a} for
the queue length up to 30  as the optimal actions for $n>30$ remain
unchanged, as stated in Corollary~\ref{corollary1} and Remark~2. In
fact, the optimal action becomes $d^*(n)=(3,4,3)$, for any $n \geq
12$. The stair-wise increase in number of working servers for the
short queue length range as shown in Fig.~\ref{fig_ex1-a} reflects
the fact that the optimal action should satisfy $d^*(n) \bm 1 \leq
n$, as stated in Proposition~\ref{pro1}. Note that
Fig.~\ref{fig_ex1-a} demonstrates that the optimal policy $d^*$
obeys the form of quasi bang-bang control defined in
Theorem~\ref{theorem3} and the number of total working servers
$||d^*(n)||_1$ is increasing in $n$, as stated in
Theorem~\ref{theorem6_monotoned}.

However, the monotone property of $d^*(n,k)$ in
Theorem~\ref{theorem8_monotoned} does not hold since
Assumption~\ref{assumption2} is not satisfied in this example. To
demonstrate this point, we change the cost rate vector to $\bm c =
(7,4,1.8)$ and keep other parameters the same as above. Using
Algorithm~\ref{algo1}, we obtain the optimal policy as illustrated
in Fig.~\ref{fig_ex1-b} after 5 iterations. We have $d^*(n)=(0,4,1)$
when $n=5$, while $d^*(n)=(2,4,0)$ when $n=6$. Therefore, we can see
that the optimal policy of group 3, $d^*(n,3)$, is not always
increasing in $n$. However, $||d^*(n)||_1$ is still increasing in
$n$ which is consistent with Theorem~\ref{theorem8_monotoned}.

\begin{figure}[htbp]
\centering \subfigure[Optimal policy with $\lambda=10$, $K=3$, $\bm
M=(3,4,3)$, $\bm \mu=(6,4,2)$, $\bm c=(7,4,3)$, $\eta^*=12.5706$.]
{\includegraphics[width=0.49\columnwidth]{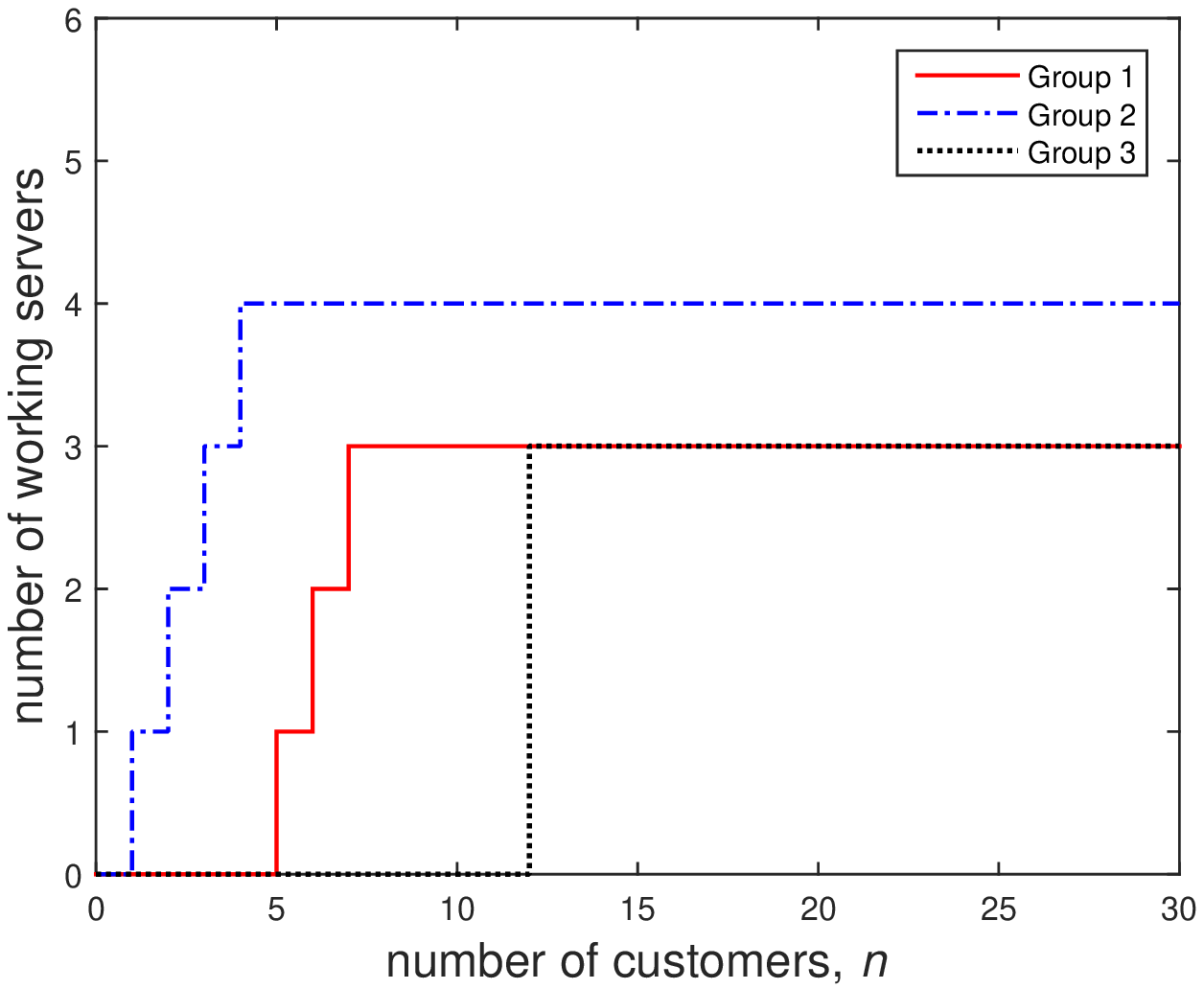}\label{fig_ex1-a}}
\subfigure[Optimal policy with $\lambda=10$, $K=3$, $\bm M=(3,4,3)$,
$\bm \mu=(6,4,2)$, $\bm c=(7,4,1.8)$, $\eta^*=12.5659$.]
{\includegraphics[width=0.49\columnwidth]{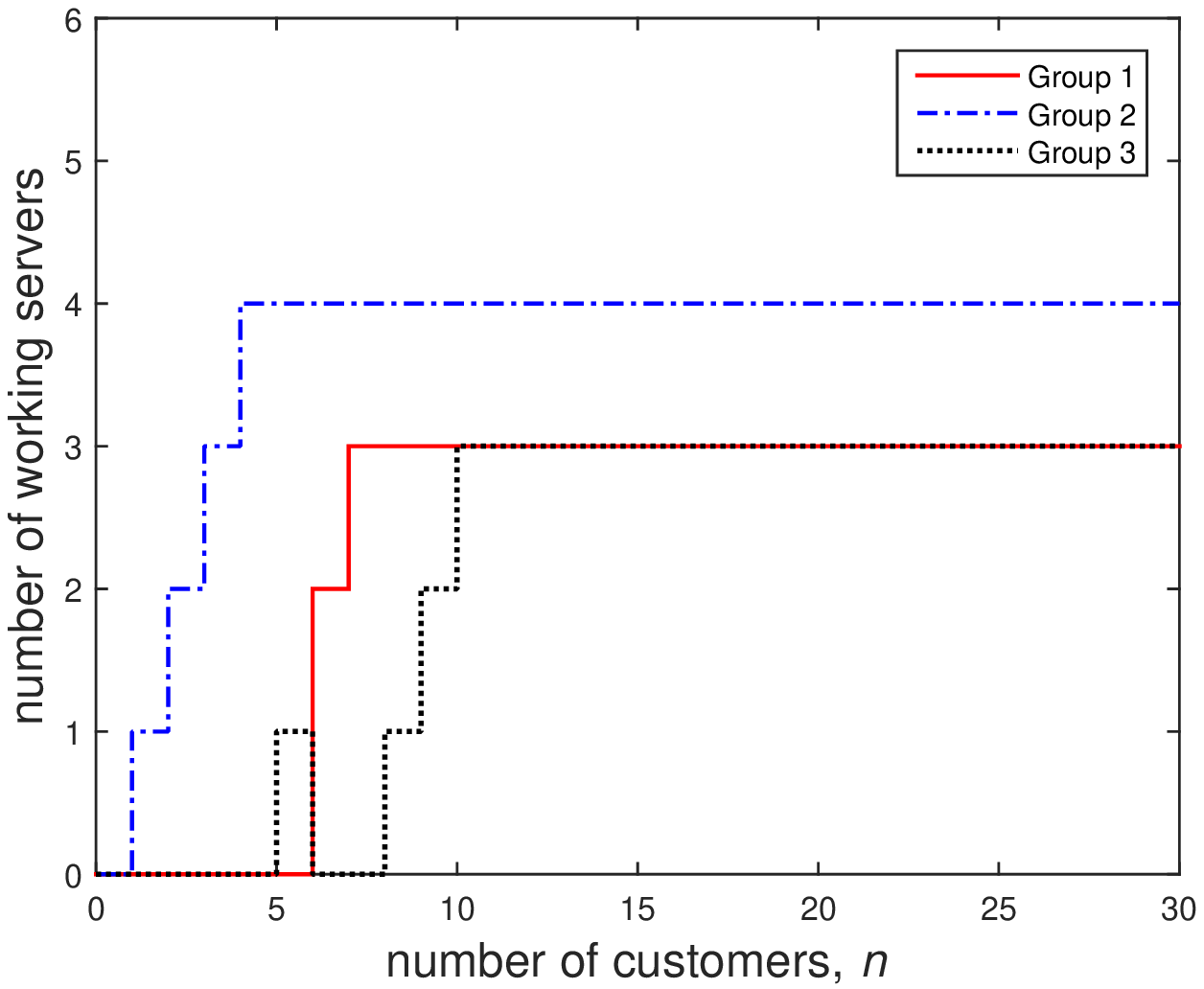}\label{fig_ex1-b}}
\caption{Optimal scheduling policies under different parameter
settings.}\label{fig_ex1b}
\end{figure}

\subsection{Example 2: A $c/{\mu}$-rule case}
We consider a system with the same set of parameters as that in the
previous example except for the operating cost rates. Now we assume
\begin{itemize}
\item Operating cost rates of groups: $\bm c = (7,8,5)$.
\end{itemize}

With these new cost rates, the descending order of $\bm \mu$ is the
same as the ascending order of $c/{\mu}$ of these groups, i.e., we
have $\mu_1>\mu_2>\mu_3$ and
$\frac{c_1}{\mu_1}<\frac{c_2}{\mu_2}<\frac{c_3}{\mu_3}$. Therefore,
Assumption 2 is satisfied and the $c/{\mu}$-rule applies to this
example. Thus, the optimal policy is a threshold vector $\bm \theta
= (\theta_1, \theta_2, \theta_3)$, as indicated by
(\ref{eq_threshold_d}). We use Algorithm~\ref{algo2} to find the
optimal threshold policy $\bm \theta^*$. From
Fig.~\ref{fig_ex2-eta}, we can see that after 5 iterations the
optimal threshold policy is found to be $\bm \theta^* = (1,9,21)$
with $\eta^* = 13.6965$. Comparing Examples~1 and 2, we note that
both algorithms take around 4 or 5 iterations to converge, at a
similar convergence speed. Algorithm~\ref{algo2} uses a threshold
policy which only has 3 variables to be determined. However, the
policy in Algorithm~\ref{algo1} is much more complex. Moreover, the
$c/{\mu}$-rule significantly simplifies the search procedure in
Algorithm~\ref{algo2}. Fig.~\ref{fig_ex2-gG} illustrates the curves
of $g^*(n)$ and $G^*(n)$, which are also consistent with the
structures stated in Theorems~\ref{theorem4} and \ref{theorem5}.

\begin{figure}[htbp]
\centering \subfigure[Average cost and queue length during
iterations.]
{\includegraphics[width=0.49\columnwidth]{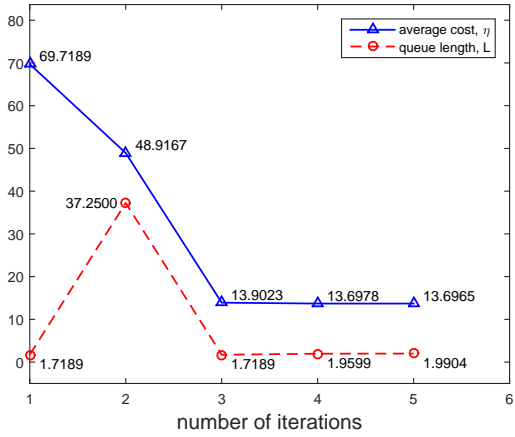}\label{fig_ex2-eta}}
\subfigure[Curves of $g^*$ and $G^*$.]
{\includegraphics[width=0.49\columnwidth]{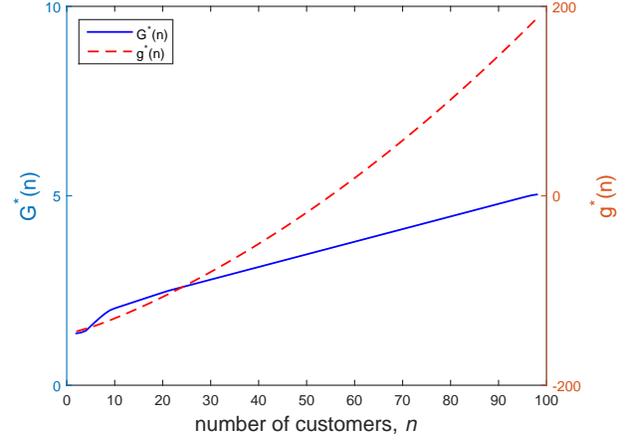}\label{fig_ex2-gG}}
\caption{Optimization procedure and curves of $g^*$ and $G^*$,
$\lambda=10$, $K=3$, $\bm M=(3,4,3)$, $\bm \mu=(6,4,2)$, $\bm
c=(7,8,5)$, $\eta^*=13.6965$.}\label{fig_ex2}
\end{figure}

\subsection{Example 3: Effect of traffic intensity}
We study the effect of traffic intensity on the optimal policy by
varying the arrival rate $\lambda$ in Example~2. Since the maximal
total service rate is $\sum_{k=1}^{K} M_k \mu_k = 40$, we examine
$0<\lambda < 40$ range for the system stability. With $\lambda = [2\
5\ 10\ 20\ 30\ 38\ 39]$, the optimal average cost $\eta^*$ and
optimal thresholds are illustrated in Fig.~\ref{fig_ex3-eta} and
Fig.~\ref{fig_ex3-theta}, respectively. As the traffic intensity
increases (the traffic becomes heavier), i.e., $\lambda \rightarrow
40$, the average cost $\eta^*$ will increase rapidly and the optimal
threshold policy $\bm \theta^*$ converges to $(1,4,8)$, which means
that servers are turned on as early as possible.

Note that the optimal threshold of the first group (its service rate
is the largest) is always $\theta^*_1 = 1$ due to the zero setup
cost, which is consistent with Theorem~\ref{theorem_th1}. It is
expected that the optimal threshold $\theta^*_1$ could be other
values if non-zero setup cost is considered.

\begin{figure}[htbp]
\centering \subfigure[Optimal average cost.]
{\includegraphics[width=0.49\columnwidth]{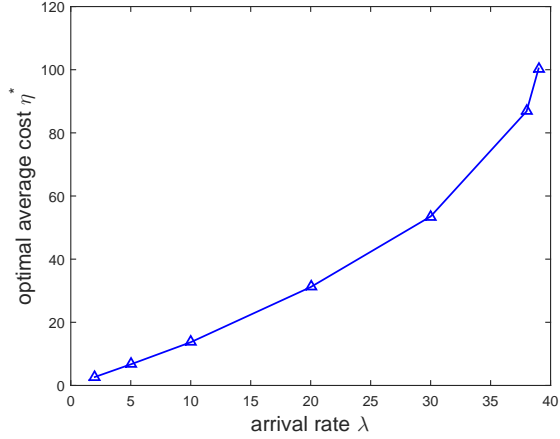}\label{fig_ex3-eta}}
\subfigure[Optimal thresholds.]
{\includegraphics[width=0.49\columnwidth]{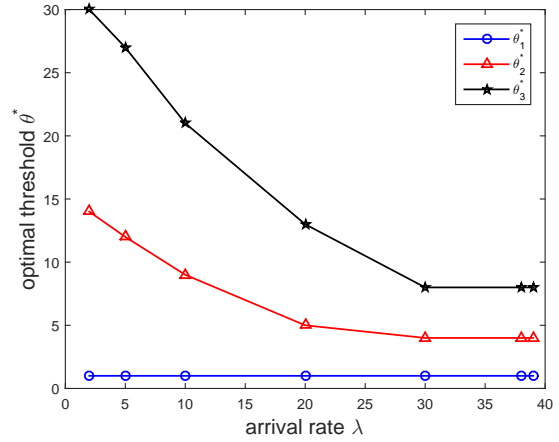}\label{fig_ex3-theta}}
\caption{Optimization results under different workloads with
$\lambda = [2\ 5\ 10\ 20\ 30\ 38\ 39]$, $K=3$, $\bm M=(3,4,3)$, $\bm
\mu=(6,4,2)$, $\bm c=(7,8,5)$.}\label{fig_ex3}
\end{figure}

\subsection{Example 4: Trade-off of costs}
For a system with the $c$/$\mu$-rule, the optimal threshold policy
depends on the dominance between the holding cost and the operating
cost. To study this effect on the optimal policy, we introduce an
operating cost weight parameter $v$. The value of $v$ reflects the
balance between the server provider's operating cost and the
customer's waiting cost. In practice, $v$ depends on the system's
optimization objective. The cost rate function (\ref{eq_f}) is
modified as below.
\begin{equation}\label{eq_fv}
f(n,\bm m) = h(n) + v \cdot \bm m \bm c.
\end{equation}
Other parameters are the same as those in Example~2. Using
Algorithm~2 and the set of operating cost weights $v = [0.1\ 0.3\
0.5\ 1\ 2\ 3]$, we obtain the minimal average costs and the
corresponding optimal threshold policies as shown in
Fig.~\ref{fig_ex4}. The curve in Fig.~\ref{fig_ex4-eta} is almost
linear, because the steady system mostly stays at states with small
queue length and the associated part of $\eta^*$ is linear in $v$.
When $v$ is small, it means that the holding cost $h(n)$ dominates
the operating cost $\bm m \bm c$ in (\ref{eq_fv}). Therefore, each
server group should be turned on earlier (smaller thresholds) in
order to avoid long queues. This explains why the optimal thresholds
are $\bm \theta^*=(1,4,8)$ both for $v=0.1$ and $v=0.3$. When $v$ is
large, the operating cost will dominate the holding cost. Thus,
except for the first group (the most efficient group), the other two
server groups are turned on only when the system is congested enough
(larger thresholds).

\begin{figure}[htbp]
\centering \subfigure[Optimal average cost.]
{\includegraphics[width=0.49\columnwidth]{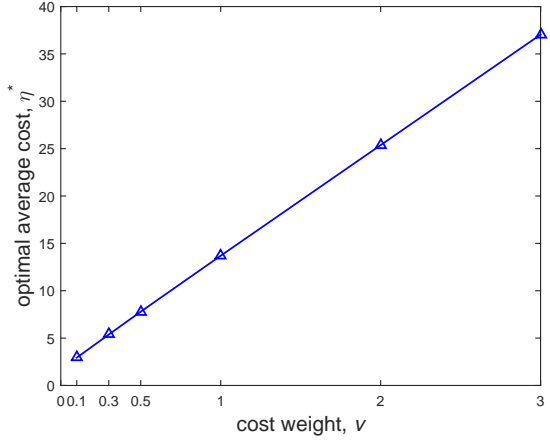}\label{fig_ex4-eta}}
\subfigure[Optimal thresholds.]
{\includegraphics[width=0.49\columnwidth]{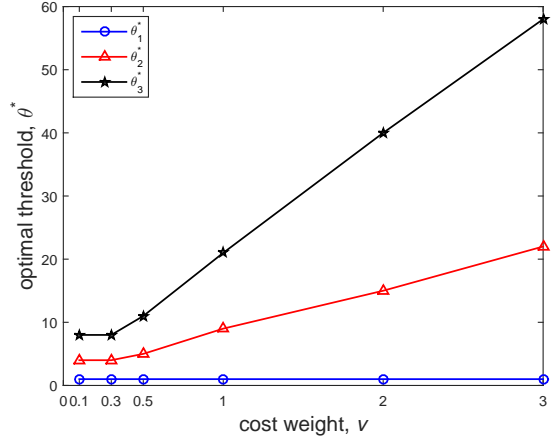}\label{fig_ex4-theta}}
\caption{Optimization results under different cost weights with $v =
[0.1\ 0.3\ 0.5\ 1\ 2\ 3]$, $\lambda=10$, $K=3$, $\bm M=(3,4,3)$,
$\bm \mu=(6,4,2)$, $\bm c=(7,8,5)$.}\label{fig_ex4}
\end{figure}

\subsection{Example 5: Model scalability}
Although all the previous examples are about small systems with only
3 server groups, our approach can be utilized to analyze large
systems with many server groups and hundreds of servers. To
demonstrate the model scalability, we consider a $c$/$\mu$-rule
system (with scale economies) where $K$ increases as $[3\ 5\ 10\ 20\
30\ 50 ]$. For the ease of implementation, we set $M_k = 3$ for all
$k$, $\bm \mu = (2,3,\cdots,K+1)$, and $\bm c = \bm \mu^{0.9}$ which
indicates a component-wise power operation of $\bm \mu$. We can
verify that this parameter setting satisfies the condition in
Assumption~\ref{assumption2}. To keep the traffic intensity at a
moderate level, we set $\lambda = 0.5 \cdot \bm M \bm \mu^T$, where
$\bm M \bm \mu^T$ is the maximal total service rate of the system.
The number of iterations of Algorithm~\ref{algo2} required for
convergence is shown in Fig.~\ref{fig_ex5} for different $K$ values.
We find that the number of iterations remains almost stable (around
3 or 4) as the system size $K$ increases. This indicates the good
scalability of our approach, namely, Algorithm 2 can be applied to a
large scale system. Note that in our model the state space remains
the same but the action space increases exponentially with $K$.
Therefore, the optimal policy structure characterized (e.g.
multi-threshold type) not only resolves the issue of \emph{infinite
state space}, but also the \emph{curse of dimensionality for action
space}.

\begin{figure}[htbp]
\centering
\includegraphics[width=0.6\columnwidth]{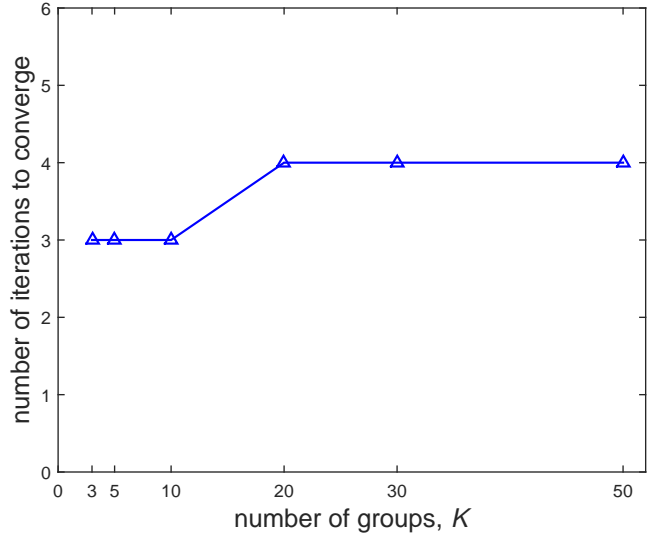}
\caption{Number of iterations needed by Algorithm~\ref{algo2} under
different problem scales with $K = [3\ 5\ 10\ 20\ 30\ 50 ]$,
$\lambda = 0.5 \cdot \bm M \bm \mu^T$, $M_k \equiv 3$, $\bm \mu =
(2,3,\cdots,K+1)$, $\bm c = \bm \mu^{0.9}$.}\label{fig_ex5}
\end{figure}

\subsection{Example 6: Robustness of the $c/\mu$-rule}
When the condition of scale economies in
Assumption~\ref{assumption2} does not hold, the optimality of the
$c/\mu$-rule is not guaranteed. Since the $c/\mu$-rule is easy to
implement, we investigate the robustness of the $c/\mu$-rule by
numerically testing several scenarios where the condition of scale
economies does not hold. For these cases, we first use Algorithm~1
to find the true optimal solution. Then, we use Algorithm~2 to find
the ``optimal" threshold policy as if the $c/\mu$-rule is
applicable, i.e., servers in group with smaller $c/\mu$ will be
turned on with higher priority. We obtain the following table to
reveal the performance gaps between the optimal policy and the
$c/\mu$-rule. The parameter setting is the same as that in
Example~1, except that we choose different cost rate vectors $\bm c$
in different scenarios.

\begin{table}[htbp]
\centering
\begin{tabular}{c|ccc}
\toprule $\bm c$    & $\eta^*$ by Algm.1 & $\hat{\eta}^*$ by Algm.2 & error\\
\midrule $[7,4,3]$  & 12.5706 & 12.5706 & 0.00\% \\
\hline   $[7,4,1.8]$& 12.5659 & 13.3287 & 6.07\% \\
\hline   $[7,4,1]$  & 11.1580 & 11.1580 & 0.00\% \\
\hline   $[8,3,1]$  & 10.0241 & 10.0615 & 0.37\% \\
\hline   $[4,3,1]$  &  8.4044 & 9.2426  & 9.97\% \\
\hline   $[18,10,3]$& 23.4844 & 23.4844 & 0.00\% \\
\bottomrule
\end{tabular}
\caption{The error effect of applying the $c/\mu$-rule when the
condition of scale economies does not hold, $\lambda=10$, $K=3$,
$\bm M=(3,4,3)$, $\bm \mu=(6,4,2)$.}\label{tab error}
\end{table}

\begin{figure}[htbp]
\centering \subfigure[Optimal policy by Algorithm~1 with $\bm
c=(7,4,1)$, it is of threshold form with $\bm \theta^*=(8,4,1)$;
``Optimal" threshold derived by Algorithm~2 is $\hat{\bm
\theta}^*=(8,4,1)$; Their error is 0.00\%.]
{\includegraphics[width=0.49\columnwidth]{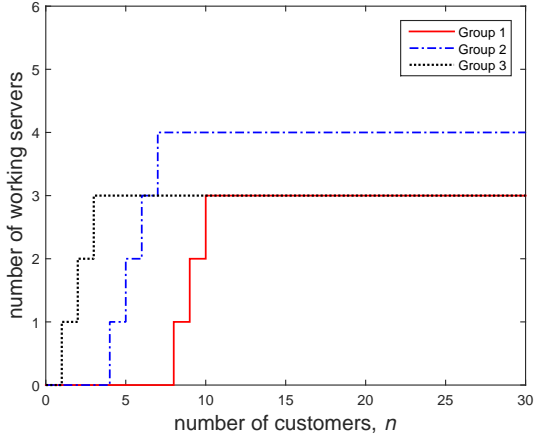}\label{fig_ex6-3}}
\subfigure[Optimal policy by Algorithm~1 with $\bm c=(8,3,1)$, it is
of threshold form with $\bm \theta^*=(11,1,5)$; ``Optimal" threshold
derived by Algorithm~2 is $\hat{\bm \theta}^*=(11,4,1)$; Their error
is 0.37\%.]
{\includegraphics[width=0.49\columnwidth]{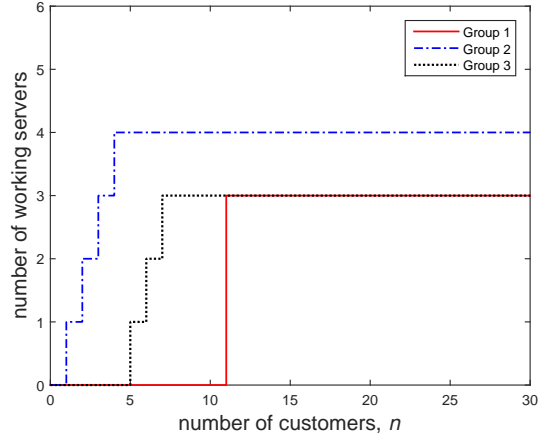}\label{fig_ex6-4}}
\subfigure[Optimal policy by Algorithm~1 with $\bm c=(4,3,1)$, it is
of threshold form with $\bm \theta^*=(1,7,4)$; ``Optimal" threshold
derived by Algorithm~2 is $\hat{\bm \theta}^*=(4,7,1)$; Their error
is 9.97\%.]
{\includegraphics[width=0.49\columnwidth]{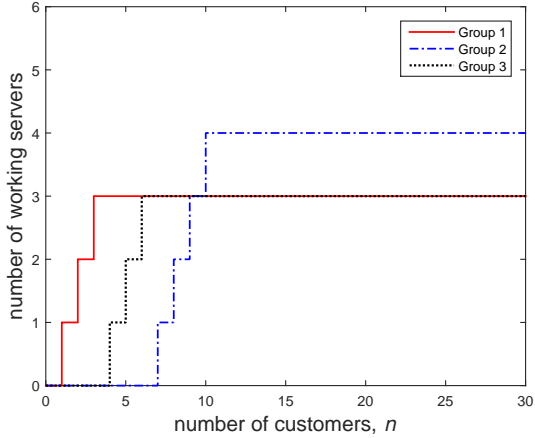}\label{fig_ex6-5}}
\subfigure[Optimal policy by Algorithm~1 with $\bm c=(18,10,3)$, it
is of threshold form with $\bm \theta^*=(11,4,1)$; ``Optimal"
threshold derived by Algorithm~2 is $\hat{\bm \theta}^*=(11,4,1)$;
Their error is 0.00\%.]
{\includegraphics[width=0.49\columnwidth]{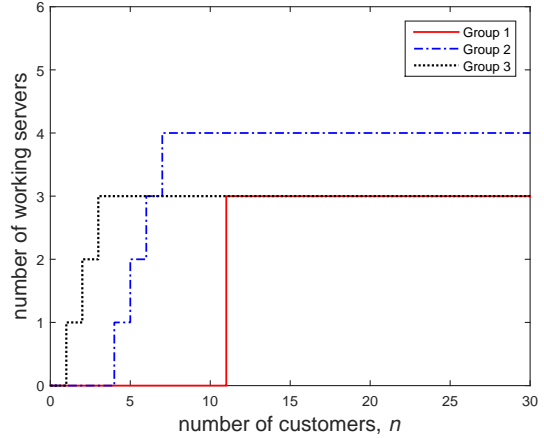}\label{fig_ex6-6}}
\caption{Solutions derived by Algorithm 1 under different operating
cost rate vectors with $\lambda=10$, $K=3$, $\bm M=(3,4,3)$, $\bm
\mu=(6,4,2)$.}\label{fig_ex6}
\end{figure}

The first three cases in Table~\ref{tab error} are designed by
changing two cost parameters from the original cost vector $\bm
c=(7,8,5)$ used in Example 2 where the scale economies condition
holds. We first change the operating cost of group 2 from 8 to 4 and
the operating cost of group 3 from 5 to 3, 1.8, and 1, respectively,
while other parameters are kept unchanged. Such parameter changes
cause the $c/\mu$ ranking sequence to change from original
$1\rightarrow 2\rightarrow 3$ to $2\rightarrow 1\rightarrow 3$,
$3\rightarrow 2\rightarrow 1$, and $3\rightarrow 2\rightarrow 1$,
respectively, for three server groups (from the fastest group 1 to
slowest group 3). Note that in Case 1 (Example 1), the $c/\mu$
rankings switch between group 1 and group 2 with group 3 ranking
unchanged so that the scale economies condition fails. However, the
$c/\mu$-rule remains the optimal. This implies that a violation of
the scale economies condition may not change the optimality of the
$c/\mu$-rule. In Case 2, the $c/\mu$ ranking sequence becomes the
reverse of the condition of scale economies and a cost gap of 6.07\%
occurs. It is interesting to see that in Case 3 a further cost
reduction of group 3 will lead to the optimality of the $c/\mu$-rule
again. For the next two cases, we keep the cost of group 3 at 1
while the costs of groups 1 and 2 are changed. It is found that the
non-optimality of the $c/\mu$-rule in these cases will cause a less
than 10\% additional cost compared to the optimal index policy.
Furthermore, it is still possible that the $c/\mu$-rule remains
optimal for the reverse order of the scale economies condition as
shown in Case 6.

Graphically, we can show how the optimal policy is different from
the $c/\mu$-rule based threshold policy. The optimal server schedule
derived by Algorithm~1 in Case 1 as shown in Fig.~\ref{fig_ex1-a} is
of threshold form with $\bm \theta^* = (5,1,8)$ which is the same as
the ``optimal" threshold derived by Algorithm~2. The optimal policy
derived by Algorithm~1 in Case 2, which is not of threshold form, is
shown in Fig.~\ref{fig_ex1-b} while the ``optimal" threshold derived
by Algorithm~2 is $\hat{\bm \theta}^* = (8,4,1)$. Such a policy
difference results in the performance degradation by 6.07\%. For
Case 3, the optimal solution derived by Algorithm~1 illustrated by
Fig.~\ref{fig_ex6-3} is of threshold form with $\bm \theta^* =
(8,4,1)$ and the ``optimal" threshold derived by Algorithm~2 is also
$\hat{\bm \theta}^* = (8,4,1)$. These two solutions coincide and
their performance error is 0. Other cases are illustrated by the
sub-figures of Fig.~\ref{fig_ex6} in a similar way.

Since $G(n)$ function plays a critical role in the optimality of the
$c/\mu$-rule and depends on multiple system parameters, we cannot
develop a pattern for the optimality of $c/\mu$-rule when the
condition of scale economies does not hold. However, from
Table~\ref{tab error} and Fig.~\ref{fig_ex6}, we observe that in
some cases, although the condition of scale economies does not hold,
the optimality of the $c/\mu$-rule still holds. In other cases, the
performance degradation caused by ``faultily" using the $c/\mu$-rule
is tolerable. This implies that the $c/\mu$-rule has a good
applicability and robustness, even for the cases where the condition
of scale economies does not hold.

\section{Conclusion}\label{section_conclusion}
In this paper, we study the service resource allocation problem in a
stochastic service system, where servers are heterogeneous and
classified into groups. Under a cost structure with customer holding
and server operating costs, we investigate the optimal index policy
(dynamic scheduling policy) which prescribes the number of working
servers in each group at each possible queue length. Using the SBO
theory, we characterize the structure of the optimal policy as a
quasi bang-bang control type. A key technical result of this work is
to establish the monotone increasing property of PRF $G^*(n)$, a
quantity that plays a fundamental role in the SBO theory. Then, the
necessary and sufficient condition and the monotone property of the
optimal policy are derived based on this property. Under an
assumption of scale economies, we further characterize the optimal
policy as the $c/\mu$-rule. That is, the servers in group with
smaller $c/\mu$ should be turned on with higher priority. The
optimality of multi-threshold policy is also proved. These
optimality structures significantly reduce the complexity of the
service resource allocation problem and resolve the issue of curse
of dimensionality in a more general heterogeneous multi-server
queueing model with infinite state space. Based on these results, we
develop the efficient algorithms for computing the optimal
scheduling policy and thresholds. Numerical examples demonstrate
these main results and reveal that the $c/\mu$-rule has a good
scalability and robustness.

A limitation of our model is that the startup and shutdown cost of
each server is assumed to be zero. The cost of customer migration
among servers is also neglected. Taking these costs into account in
our model can be a future research topic. Moreover, we assume linear
operating costs in this paper. It would be interesting to extend our
results to a more general operating cost structure. Asymptotically
extending to the scenario of many servers in a networked setting
under the fluid regime can also be another future research
direction.

\section*{Acknowledgement}
The authors would like to present grates to Prof. Xi-Ren Cao, Prof.
Christos Cassandras, Prof. Jian Chen, and Prof. Leandros Tassiulas
for their valuable discussions and comments.

\end{document}